\documentclass[a4paper,12pt,reqno,oneside]{amsart}
\usepackage{setspace} 
\usepackage{typearea} % European margins

\usepackage{amscd,amsmath,amssymb,verbatim}
\usepackage{xr-hyper}
\usepackage{graphicx}
\usepackage{ifthen}
\usepackage{paralist,cite}
\usepackage[colorlinks,backref,final,bookmarksnumbered,bookmarks]{hyperref}
\usepackage[backrefs]{amsrefs}
\usepackage{nicematrix}
\usepackage{tikz-cd}
\usepackage{accents}
\usepackage{enumitem}

\usepackage[utf8]{inputenc}
\usepackage[T1,T2A]{fontenc}

\newcommand{\arxiv}[2][]{\ifthenelse{\equal{#1}{}}
{\href{http://arxiv.org/abs/#2}{\tt arXiv:#2}}
{\href{http://arxiv.org/abs/math/#2}{\tt arXiv:math.#1/#2}}}

\theoremstyle{plain}
\newtheorem{maintheorem}{Theorem}
 
\newtheorem{theorem}{Theorem}[section]
\newtheorem{lemma}[theorem]{Lemma}
\newtheorem{corollary}[theorem]{Corollary}
\newtheorem*{corollary*}{Corollary}
\newtheorem{proposition}[theorem]{Proposition}
\newtheorem{problem}[theorem]{Problem}

\theoremstyle{definition}

\newtheorem{example}[theorem]{Example}

\newtheoremstyle{remark}
{}{}{}{}{\itshape}{}{ }{\thmname{#1}\thmnumber{ \itshape #2.}}
\theoremstyle{remark}
\newtheorem{remark}[theorem]{Remark}

\def\N{\mathbb{N}} 
\def\R{\mathbb{R}} 
\def\Z{\mathbb{Z}} 

\def\F{\mathcal{F}} 
\def\G{\mathcal{G}} 
 
\def\C{\mathcal{C}}
\def\K{\mathcal{K}}
\def\M{\mathcal{M}}
\def\NN{\mathcal{N}}

\def\U{\mathcal{U}}
\def\V{\mathcal{V}}
\def\pro{\mathsf{pro\text{-}}}
\def\ind{\mathsf{ind\text{-}}}
\def\inv{\mathsf{inv\text{-}}}
\def\dir{\mathsf{dir\text{-}}}
\def\Top{\mathsf{Top}}
\def\Ho{\mathsf{Ho}}
\def\Sh{\mathsf{Sh}}
\def\Ssh{\mathsf{Ssh}}
\def\Fish{\mathsf{Fish}}
\def\Ash{\mathsf{Ash}}
\def\Ap{\mathsf{Ap}}
\def\Gap{\mathsf{Gap}}
\def\x{\times}
\def\but{\setminus} 
 
\def\eps{\varepsilon} 
\def\phi{\varphi} 
\def\To{\Rightarrow}

\def\xr#1{\xrightarrow{#1}} 
 \renewcommand{\:}{\colon}

\DeclareMathOperator*{\colim}{colim}

\DeclareMathOperator{\Int}{Int} \DeclareMathOperator{\id}{id}
 \DeclareMathOperator{\Cl}{Cl}

\newcommand{\cev}[1]{\accentset{\leftarrow}{#1}}
\renewcommand{\vec}[1]{\accentset{\rightarrow}{#1}}

\def\tph#1{\raise2.5pt\hbox{\the\textfont1\char"7F}\!\!#1}
\def\tpm#1{\raise0pt\hbox{\the\textfont1\char"7F}\!#1}
\def\tpl#1{\lower1.5pt\hbox{\the\textfont1\char"7F}\!#1}

\def\bydef{\mathrel{\mathop:}=}

\DeclareSymbolFont{bskadd}{U}{bskma}{m}{n}
\DeclareFontFamily{U}{bskma}{\skewchar\font130 }
\DeclareFontShape{U}{bskma}{m}{n}{<->bskma10}{}
\DeclareMathSymbol{\varlrttriangle}{\mathord}{bskadd}    {"E4}

\makeatletter
\newcommand*\nullseq{\mathop{\mathpalette\@biguoperator{\bigsqcup}}}
\newcommand*\@biguoperator[2]{\ooalign{\hidewidth$#1$\raisebox{1pt}{\scriptsize$\star$}\hidewidth\cr$#1#2$\cr}}
\makeatother
\begin{document}

\title{Fine shape. I}
\author{Sergey A. Melikhov}
%\date{\today}
\address{Steklov Mathematical Institute of Russian Academy of Sciences,
ul.\ Gubkina 8, Moscow, 119991 Russia}
\email{melikhov@mi-ras.ru}
%\thanks{Supported by ...}

\begin{abstract} 
We introduce and develop fine shape, which has a very simple definition and aims to supersede 
all previously known shape theories for metrizable spaces.

The problem with known shape theories of metrizable spaces is illustrated by the following bizarre 
situation.
\v Cech cohomology is an invariant of {\it shape}, and a fortiori of {\it strong shape}.
But Steenrod--Sitnikov homology is not shape invariant (already for compacta) and has not been proved 
to be strong shape invariant (in more than 40 years).
Worse yet, there is an ordinary homology theory that is strong shape invariant by design, but it cannot be
computed in ZFC for simplest non-compact non-ANRs (such as the disjoint union of countably many copies 
of the one-point compactification of the countable discrete space).
On the other hand, Steenrod--Sitnikov homology is an invariant of {\it antishape} 
(=compactly generated strong shape), and a fortiori of {\it strong antishape}.
However, \v Cech cohomology is not antishape invariant (already for ANRs) and has not been proved
to be strong antishape invariant.
And there is an ordinary cohomology theory that is strong antishape invariant by design, but it cannot be
computed in ZFC for simplest non-compact non-ANRs.

Even though strong shape and strong antishape differ from each other by exchanging direct and inverse limits, 
we show that their natural ``corrections'' (taking into account a topology on the indexing sets) coincide 
for all metrizable spaces.
This common ``correction'', called fine shape, is much simpler than the original theories and has both 
\v Cech cohomology and Steenrod--Sitnikov homology as its invariants.
For ANRs fine shape coincides with homotopy, for compacta with strong shape, and for 
locally compact separable metrizable spaces --- with strong antishape.

We prove that a (co)homology theory is fine shape invariant if and only if it satisfies 
the map excision axiom.
\end{abstract}

\maketitle
%\tableofcontents
\section{Introduction}

Classical homotopy theory and classical homotopy invariants: homotopy groups, singular homology and 
cohomology --- work well for spaces that are homotopy equivalent to CW complexes, such as polyhedra%
\footnote{That is, simplicial complexes with the metric topology (see \cite{M00} for a detailed treatment).}
and ANRs (as indicated, in particular, by the Whitehead and Hurewicz theorems).
But they often exhibit highly pathological behavior when applied to more general spaces 
(see, for instance, \cite{M1}*{Theorem 1.1 and Example 5.6}).
Better suited to study more general spaces are ``controlled'' invariants, which agree with classical 
invariants on polyhedra and treat more general spaces by looking at their polyhedral approximations 
(such as nerves of open covers).
Or if you prefer, they are supposed to agree with classical invariants on ANRs and treat more general 
metrizable spaces by looking at their ANR neighborhoods (such as open neighborhoods in a normed vector space).
Of course, this describes ``controlled'' only as an informal concept.
Nevertheless, it is well-known that for metrizable spaces there is just one reasonable ``controlled'' ordinary
cohomology theory --- {\it \v Cech cohomology},%
\footnote{Or equivalently \cite{Sk5} Alexander--Spanier cohomology \cite{Sp} or sheaf cohomology with constant coefficients.}
and just one reasonable ``controlled'' ordinary homology theory --- {\it Steenrod--Sitnikov homology}.%
\footnote{Steenrod--Sitnikov homology is the direct limit of Steenrod homology of compacta \cite{Ste}, \cite{Si}, and the latter
has a number equivalent descriptions: Kolmogorov's, with chains as functions on tuples of subsets \cite{Kol}, \cite{Md}; 
Massey's, with chains defined in terms of cochains, which are finitely-valued functions on tuples of points \cite{Mas} 
(see also \cite{Sk2}); Milnor's, with cochains defined in terms of singular chains of function spaces \cite{Mi1}*{\S4} 
(see also \cite{Sk2}); Deheuvels' in terms of cosheaves \cite{Deh1}, \cite{Deh2}; and, in the case of finitely generated 
coefficients, also the one by Borel and Moore in terms of sheaves \cite{BM}, \cite{Sk69}.}
This view is supported, in particular, by the uniqueness theorems of Milnor--Petkova \cite{Mi1}, \cite{Pe1}*{Theorem 7} 
(see also \cite{M00}*{Theorem \ref{book:uniqueness2}}) and Bacon \cite{Bac2} and by the Sitnikov duality 
(see \cite{M00}*{Theorem \ref{book:alex duality}}).

``Controlled'' homotopy theories are called {\it shape theories}.
It is well known that for compacta (=compact metrizable spaces) there is just one reasonable shape theory, known 
as {\it Steenrod homotopy} or {\it strong shape}.
It was constructed in a Princeton dissertation by D. Christie, which was written under the supervision 
of Lefschetz and published in Transactions of the AMS in 1944 \cite{Ch}.
Yet it remained unnoticed until mid-70s, when it was rediscovered as a ``corrected'' version of Borsuk's 
{\it shape} \cite{Bo1} by several groups of authors, and reformulated in a variety of equivalent approaches 
\cite{Qu}, \cite{Gr}, \cite{EH}, \cite{Ba1}, \cite{Ba2}, \cite{DS1}, \cite{KO}, \cite{KO2}, \cite{Ko}, 
\cite{Fe}*{\S5}; see also \cite{EH2} (but beware of errors in \S2), \cite{Dy}, \cite{DS2}, \cite{Ca}.
One way to define the strong shape category for compacta \cite{EH}, \cite{KO}, \cite{DS2}, closely 
related to Chapman's characterization of shape, is that a strong shape morphism between compacta $X$ and $Y$,
embedded as Z-sets in the Hilbert cube $I^\infty$, is a proper homotopy class of proper maps 
$I^\infty\but X\to I^\infty\but Y$. (This turns out to be independent of the embeddings.)
A detailed treatment of strong shape of compacta can be found in \cite{M1}, which contains simplified proofs
of a number of previous results, as well as a further development of the area, including first correct 
proofs of the Whitehead-type and Hurewicz-type theorems in terms of Steenrod homotopy groups and 
Steenrod homology groups.

The story for non-compact spaces used to be much more complicated.
Quite a few different shape theories of non-compact spaces have been constructed in the last 50 years ---
but none of them was entirely satisfactory.
These will be discussed in some detail below, but roughly speaking, at first there were theories (such as 
Fox's shape and compactly generated shape) which are obviously ``defective'' as measured by $\lim^1$.
Then there were other theories (such as strong shape) which lack those obvious ``defects'', and coincide
with Steenrod homotopy for compacta, but become very complex and intractable for non-compact spaces, both 
in practical terms and in terms of what they are able to say without adding new axioms to ZFC.
Besides, most of these shape theories failed (or were never proved) to have both Steenrod--Sitnikov 
homology and \v Cech cohomology as their invariants.

In the present paper we introduce a new shape theory of metrizable spaces, which we call ``fine shape''.%
\footnote{Y. Kodama and his students have used ``fine shape'' to call what is now better known as strong shape of 
compacta \cite{KO}.
As far as I know, they have only applied this terminology in the compact case, so it appears to be free in 
the non-compact case. 
Also, Kodama's geometric approach to strong shape appears to be rather close in spirit to the approach of 
the present paper.}
We will see that it is not only free of the deficiencies of the previous theories, but also easier to define.
The results of the present paper and its sequels%
\footnote{In particular, part II proves an abelian Whitehead-type theorem for local compacta (=locally compact 
separable metrizable spaces), and part III proves that every Polish space is fine shape equivalent to the limit 
of an inverse sequence of simplicial maps between metric simplicial complexes.}
suggest that it is just that one reasonable shape theory which was no longer hoped to exist.

\subsection{Fine shape}
There are several ways to define the fine shape category, but one of them is particularly simple.

Given metrizable spaces $X$ and $Y$, let $M$ and $N$ be absolute retracts containing them as closed Z-sets
(we recall in \S\ref{z-sets} what this means and why they exist).
A map $M\but X\to N\but Y$ is called {\it $X$-$Y$-approaching} if it sends every sequence of points which has 
a cluster point in $X$ into a sequence which has a cluster point in $Y$.
Then {\it fine shape classes} between $X$ and $Y$ are the classes of $X$-$Y$-approaching maps $M\but X\to N\but Y$
up to $X$-$Y$-approaching homotopies (that is, homotopies that are $(X\x I)$-$Y$-approaching as maps 
$(M\but X)\x I\to N\but Y$).
It turns out that these classes do not depend on the choice of $M$ and $N$ (see \S\ref{fsm-section}).
The {\it fine shape category} has metrizable spaces as objects and their fine shape classes as morphisms.

\begin{example}\label{approaching example}
If $F\:M\to N$ is a continuous map such that $F^{-1}(Y)=X$, then its restriction $f\:M\but X\to N\but Y$ is 
easily seen to be an $X$-$Y$-approaching map.
\end{example}

\begin{example} When $M$ and $N$ are compact, a map $M\but X\to N\but Y$ is $X$-$Y$-approaching 
if and only if it is proper.
\end{example}

\begin{remark} F. Ancel pointed out a relation of $X$-$Y$-approaching maps with strongly continuous relations 
in the sense of J. Cannon \cite{Can}.
See Remark \ref{relations} for the details.
\end{remark}

To get a better understanding of fine shape it helps to look at its other definitions.

\begin{maintheorem}\label{3-defs}
For a map $f\:M\but X\to N\but Y$ the following are equivalent:
\begin{enumerate}
\item $f$ is an $X$-$Y$-approaching map;
\item if $U\subset N\but Y$ is such that $U\cup Y$ is an open neighborhood of $Y$ in $N$,
then $f^{-1}(U)\cup X$ is an open neighborhood of $X$ in $M$;
\item if $K\subset M\but X$ has compact closure in $M$, then $f(K)$ has compact closure in $N$
and $f|_K\:K\to f(K)$ is a proper map.
\end{enumerate}
\end{maintheorem}
The equivalence of (2) and (3) looks very surprising: why would a condition on open neighborhoods
relate at all to a condition on compacta?
While the proof is not difficult (see Proposition \ref{map duality1}), this appears to be 
a totally new kind of duality.
As discussed in \S\ref{filtrations}, it is present already on the level of spaces (without any maps coming into play); 
namely, the families $\{K\subset X\but M\mid K\text{ has compact closure in $M$}\}$ and 
$\{U\subset M\but X\mid U\cup X\text{ is an open neighborhood of $X$ in $M$}\}$ of subsets of $M\but X$ determine each other.
This duality also turns out to provide a new approach to axiomatic homology and cohomology \cite{M-II}.

\begin{maintheorem} \label{main2a}
(a) Fine shape is weaker than homotopy (this is trivial, see \S\ref{fsm-section}).

(b) For ANRs fine shape coincides with homotopy (Theorem \ref{ANR}).

(c) For compacta fine shape coincides with strong shape (Proposition \ref{compact-ss}).
\end{maintheorem}

In more detail, (a) asserts that there is a functor from the homotopy category of metrizable spaces
to the fine shape category which is the identity on objects.%
\footnote{In other words, every map $X\to Y$ determines a fine shape morphism $X\to Y$ in a way that is 
compatible with composition and the identity and so that homotopic maps yield the same fine shape morphism.}
A consequence of this is that homotopy equivalent metrizable spaces are fine shape equivalent.

\begin{maintheorem}\label{main5}
A (co)homology theory on closed pairs of metrizable spaces satisfies the map excision axiom 
if and only if it is fine shape invariant.
\end{maintheorem}

See \S\ref{mex-section} for the statement of the map excision axiom and for the precise meaning of
``fine shape invariance''.
The compact case of Theorem \ref{main5} is due to Mrozik \cite{Mr2}.

Let us note that fine shape invariance is a strong form of homotopy invariance; 
but Theorem \ref{main5} says that it is also a strong form of the excision axiom.

\begin{corollary} \label{homology0} 
Steenrod--Sitnikov homology and \v Cech cohomology are invariants of fine shape.
\end{corollary} 

\begin{example} (a) The universal cover $\tilde X$ of $X\bydef S^n\vee S^1$, $n>1$, is not fine shape equivalent 
to any compactum $K$.
Indeed, $H^n(\tilde X)\simeq\prod_{i=1}^\infty\Z$, but $H^n(K)$ is countable, since it is isomorphic to 
$\colim_{i\in\N} H^n(Q_i)$, where $Q_i$ are compact polyhedra.

(b) The metric wedge $X\bydef \bigvee_{i=1}^\infty S^n$ (also known as the $n$-dimensional Hawaiian Earring) 
is not fine shape equivalent to any separable ANR $P$.
Indeed, $H_n(X)\simeq\prod_{i=1}^\infty\Z$, but $H_n(P)$ is countable since $P$ is homotopy equivalent
to a separable polyhedron.
\end{example}

The proof of Theorem \ref{main5} employs the following result, which appears to be 
new even in the case where $X$ and $Y$ are compacta.

\begin{maintheorem} \label{main6}
Let $M$ and $N$ be metrizable spaces and let $X$ and $Y$ be their closed 
homotopy negligible subsets.
If $M\but X$ and $N\but Y$ are $X$-$Y$-approaching homotopy equivalent, then $X$ and $Y$ 
are fine shape equivalent.
\end{maintheorem}

Of course, when $M$ and $N$ are absolute retracts, this follows from the definition of fine shape.
The point is that $M$ and $N$ can be arbitrary metrizable spaces (see \S\ref{approaching cat}).
This provides a convenient tool for constructing fine shape equivalences:

\begin{corollary} \label{glass+cone} Let $A$ be a closed subset of a metrizable space $X$.

(a) If $A$ is compact, $X/A$ is fine shape equivalent to $X\cup_A CA$, where $CA$ is the cone. 

(b) $(X\x I)\but Q$ is fine shape equivalent to $(X\x\{0\}\cup A\x I)\but Q$ for any $Q\subset\Int A\x I$.
\end{corollary}

When $X$ is compact, both (a) and (b) are well-known (see respectively \cite{M00}*{Proposition \ref{book:ssh-cone}} 
and references there and \cite{Ca}*{Corollary 1.6}).

We will prove a generalization of (a) involving a mapping cylinder (Lemma \ref{fish-MC}), which in turn will be 
a key step in the proof of the ``if'' part of Theorem \ref{main5}.

\begin{proof}[Proof. (b)]
Let $U_\eps$ be the closed $\eps$-neighborhood of $A$ in $X$. 
For each $\eps>0$ clearly $X\x I$ deformation retracts onto $X\x\{0\}\cup V_\eps$, 
where $V_\eps=\bigcup_{t\in I} U_{\eps(1-t)}\x\{t\}$.
Moreover, this deformation retraction continuously depends on $\eps$, and so yields a deformation retraction $r_t$ of
$X\x I\x(0,1]$ onto $X\x\{0\}\x(0,1]\cup\bigcup_{\eps\in(0,1]}V_\eps\x\{\eps\}$.
Clearly $r_1$ is an $(X\x I)$-$(X\x\{0\}\cup A\x I)$-approaching map and $r_t$ is 
an $(X\x I)$-$(X\x I)$-approaching homotopy.
It remains to remove $Q\x(0,1]$ and apply Theorem \ref{main6}.
\end{proof}

\begin{proof}[(a)]
First we note that $Y\bydef X\x I\big/A\x\{1\}$ deformation retracts onto a copy of $X/A$ (namely, onto $X\x\{1\}\big/A\x\{1\}$)
and at the same time contains a copy of $X\cup_A CA$ (namely, $Z\bydef X\x\{0\}\cup A\x I\big/A\x\{1\}$).
The deformation retraction $r_t$ constructed in the proof of (b) descends to a deformation retraction $q_t$ of $Y\x(0,1]$ onto 
$X\x\{0\}\x(0,1]\cup\bigcup_{\eps\in(0,1]}V_\eps\x\{\eps\}\big/A\x\{1\}\x(0,1]$.
Since $A$ is compact, $q_1$ is a $Y$-$Z$-approaching map and $q_t$ is a $Y$-$Y$-approaching homotopy.
It remains to apply Theorem \ref{main6}.
\end{proof}

\begin{example} When $A$ is non-compact, the metric quotient $X/A$ (see \cite{M00}*{\S\ref{book:metric quotient}}) need not be 
fine shape equivalent to $X\cup_A CA$, where $CA$ is the metric cone (see \cite{M00}*{\S\ref{book:metric cone}}).%
\footnote{Incidentally, let us note that $CA$ does not depend on the choice of the metric on $X$, but $X/A$ does.}
Indeed, let $X=\{(n,\frac1n)\mid n\in\N\}\cup [1,\infty)\subset\R^2$ with the Euclidean metric and let $A=[1,\infty)$.
Then the metric quotient $X/A$ is homeomorphic to $\{\frac1n\mid n\in\N\}\cup\{0\}$.
However, $H^0(X/A)$ is countable since $X/A$ is compact, whereas $H^0(X\cup_A CA)$ is uncountable, since $X\cup_A CA$ 
contains infinitely many connected components.
\end{example}

As a byproduct of the proof of the ``only if'' part of Theorem \ref{main5} we obtain the following
unexpected result (Theorem \ref{mex-supports}):

\begin{maintheorem} \label{homology-theories}
Let $(X,A)$ and $(Y,B)$ be closed pairs of metrizable spaces and suppose that
$h\:X\but A\to Y\but B$ is a homeomorphism such that $h$ is an $A$-$B$-approaching map
and $h^{-1}$ is a $B$-$A$-approaching map.
If $H$ is a (co)homology theory on closed pairs of metrizable spaces satisfying 
the map excision axiom, then $H_*(X,A)\simeq H_*(Y,B)$ (resp.\ $H^*(X,A)\simeq H^*(Y,B)$).
\end{maintheorem}

This suggests an entirely new approach to homology theories, which will hopefully be developed in 
a subsequent paper by the author.

\subsection{Shape, antishape and UCOSS}

By Theorem \ref{3-defs} the following three naturally available approaches to fine shape: 
\begin{enumerate}
\item[($\nu$)] one based on neighborhoods, 
\item[($\kappa$)] one based on compacta and 
\item[($\kappa\nu$)] one based simultaneously on neighborhoods and on compacta 
\end{enumerate}
all yield the same result.
Borsuk \cite{Bor} has observed that in his original setting (based on countable ``fundamental sequences'') 
the corresponding three shape theories all differ from each other on metrizable spaces, even 
though they all coincide on compacta.
This observation of Borsuk dates from 1972, when, in his own words, ``the attempts to extend [the theory of shape 
of compacta] onto arbitrary metrizable spaces [we]re only at the beginning'' \cite{Bor}.
However, subsequent literature on the subject only aggravated the situation by producing still more of
inequivalent definitions (see e.g.\ \cite{San}).

Starting from the early 70s, the basic dichotomy was between theories of type ($\nu$), exemplified by {\it shape} 
(see \cite{Fox1}, \cite{DS}, \cite{MS}, \cite{M00}*{\S\ref{book:shape}}) and theories of type ($\kappa$), exemplified by
{\it compactly generated shape} (see \cite{SS}, \cite{WD} and references there).
It should be noted that both shape and compactly generated shape are obviously defective shape theories,  
because Steenrod--Sitnikov homology is not an invariant of either of them (already for compacta) and
their restrictions to compacta differ from strong shape.

Shape and compactly generated shape are not dual to each other from the viewpoint of the duality between 
inverse and direct limits. 
Indeed, the morphisms of compactly generated shape may be identified with those of $\ind(\pro\C)$, where 
$\C$ is the homotopy category of compact polyhedra (see e.g.\ \cite{M00}*{\S\ref{book:pro} and \S\ref{book:ind}} 
concerning the $\pro$ and $\ind$categories).
The morphisms of the dual theory, which might be termed ``shape based on compactly generated homotopy'', 
are those of $\pro(\ind\C)$.
Here ``compactly generated homotopy'' (called ``weak homotopy'' by Adams \cite{Ad} and ``finite-homotopy'' 
in \cite{Mal}) considers two maps between polyhedra to be equivalent if their restrictions to every compact 
subset%
\footnote{Or equivalently to every finite subcomplex of a fixed triangulation 
(see \cite{M00}*{Theorem \ref{book:poly-cofinal}}).}
are homotopic.
The difference between homotopy and compactly generated homotopy is captured by $\lim^1$ in cohomology 
of non-compact polyhedra (see e.g.\ \cite{M00}*{Example \ref{book:phantom}}), and so is the difference
between shape and shape based on compactly generated homotopy.

Dual to shape is then compactly generated {\it strong} shape, which we will call {\it antishape} for brevity%
\footnote{It would be more natural to call it ``coshape'' but this word already has strong connotations with 
a less remarkable sort of theories, related to weak homotopy type.
One could also call antishape ``Steenrod--Sitnikov homotopy''.}
(see \cite{M00}*{\S\ref{book:antishape}} for a treatment).
The difference between compactly generated shape and antishape is likewise captured by $\lim^1$ in 
Steenrod homology of compacta.
Antishape is an obviously defective shape theory, because \v Cech cohomology is not
its invariant (already for polyhedra) and its restriction to polyhedra differs from homotopy.

\begin{maintheorem} \label{main1a} 
(a) Fine shape is stronger than shape (Theorem \ref{shape}).

(b) Fine shape is stronger than antishape (Theorem \ref{antishape}).
\end{maintheorem}

Theorem \ref{main1a}(a) immediately%
\footnote{To be precise, (1) follows from (4) (see \cite{M00}*{\S\ref{book:cech-section}}), which in turn follows from Theorem \ref{main1a}(a).
And (2) and (3) follow from (5), which in turn follows from Theorem \ref{main1a}(a) (see \cite{M00}*{\S\ref{book:fun-example2}}).
(Actually, there is no much difference between (4) and (5).)
And (6) follows from Theorem \ref{main1a}(a) (see \cite{M00}*{\S\ref{book:quasicomponents}}).}
yields

\begin{corollary} \label{cohomology} The following are invariants of fine shape:
\begin{enumerate}
\item \v Cech cohomology;
\item stable cohomotopy;
\item complex K-theory (as given by a spectrum);
\item more generally, every shape invariant cohomology theory;
\item the homotopy set $[X,Q]$ for a fixed ANR $Q$, regarded as a functor of $X$;
\item for a separable metrizable $X$, the space of quasi-components of $X$.
\end{enumerate}
\end{corollary}

Theorem \ref{main1a}(b) immediately%
\footnote{To be precise, (1), (2) and (3) follow from (4), which in turn follows from Theorem \ref{main1a}(b)
(see \cite{M00}*{Theorem \ref{book:homology-ext}}).
And (5) similarly follows from Theorem \ref{main1a}(b) (see \cite{M00}*{Theorem \ref{book:sit-extension}}).}
yields

\begin{corollary} \label{homology}
The following are invariants of fine shape:
\begin{enumerate}
\item Steenrod--Sitnikov homology;
\item Steenrod--Sitnikov bordism;
\item K-homology of Brown--Douglas--Fillmore;
\item more generally, every homology theory that has compact supports and is strong shape 
invariant on compacta;
\item Steenrod--Sitnikov homotopy groups.
\end{enumerate}
\end{corollary} 

Much (but not all) of Corollaries \ref{cohomology} and \ref{homology} can also be deduced from Theorem \ref{main5}.

\begin{example}
The Baire space $\N^\infty$ is not fine shape equivalent to any separable metrizable space $X$ that is 
a countable union of compacta.
Indeed, the space of quasi-components $Q(\N^\infty)$ is clearly homeomorphic to $\N^\infty$ and hence is not 
a countable union of compact subsets.
On the other hand, since $X$ is a countable union of compacta, $Q(X)$, which is Hausdorff but not necessarily 
metrizable, is a countable union of their images, which are compact.
\end{example}

A remarkable shape theory of type ($\kappa\nu$) was constructed by F. Bauer \cite{Ba0}.
It is a blend of shape and antishape.
More precisely, Bauer also blended them with stable homotopy and called the resulting theory
``compact-open strong shape'', but we are interested in the unstable version of his construction, 
which will thus be abbreviated {\it UCOSS}.
By design, UCOSS is stronger than shape and than antishape.
In particular, both Steenrod--Sitnikov homology and \v Cech cohomology are its invariants.
It can be shown by combining the proofs of parts (a) and (b) of Theorem \ref{main1a} that UCOSS is 
weaker than fine shape.
For polyhedra UCOSS coincides with homotopy, and for compacta with strong shape (see \cite{Ba0}).
Thus shape and antishape do cancel the most obvious ``defects'' of each other in Bauer's theory.
But the fact that they are just lumped together in an ad hoc manner makes the UCOSS theory look
artificial, and sows the suspicion that it might retain some less obvious, hidden ``defects''.
Anyone who considers using it would inevitably wonder if a better solution exists.
(And it does, as we have seen.)

\subsection{Strong shape and strong antishape}

Strong shape of non-compact spaces arises as the most obvious way of ``correcting'' shape by
the familiar method of introducing coherence relations, and due to this it has a number of 
equivalent definitions (see \cite{Mard}, \cite{Ca} and references there).
Unfortunately, all of them are either indirect and non-constructive, or very tedious, as one has to deal
with $n$-dimensional coherence conditions for all $n$ (i.e.\ essentially with higher categories).
Working with higher categories may sound exciting, but in practice it means that either the proof 
of every minor lemma takes several pages (as in the book \cite{Mard} on strong shape), or a lot has 
to be glossed over, with increasing risk of mistakes (as in the erroneous papers on higher categories 
discussed in \cite{Vo}).

The definition of strong shape can be immensely simplified for a class of spaces called coronated ANRs.
A {\it coronated ANR} is a metrizable space that contains a compactum whose complement is an ANR.
Coronated ANRs include not only compacta and ANRs, but also, for instance, complements of 
all locally compact subsets of spheres \cite{M-V}.
It is proved in \cite{M-V} that a metrizable space is a coronated ANR if and only if it 
admits a sequential resolution where the images of all the bonding maps $p^i_j$ are ANRs.
Due to this, strong shape morphisms between coronated ANRs admit a much simpler description, involving 
only the 1-dimensional coherence conditions (see \cite{M00}*{\S\ref{book:seq-ss}} for the details).

\begin{maintheorem} \label{main1b} 
Fine shape is stronger than strong shape (Theorem \ref{strongshape}).
\end{maintheorem}

While this implies Theorem \ref{main1a}(b), we first prove Theorem \ref{main1a}(b)
and then build on that proof to prove Theorem \ref{main1b}.

Just like shape can be ``improved'' to strong shape, antishape can be ``improved'' to {\it strong antishape} by 
introducing coherence conditions.
We carry out this procedure in detail in the case of {\it local compacta} (=locally compact separable metrizable spaces), 
in which only $1$-dimensional coherence conditions are needed (see \S\ref{sash}).
For arbitrary metrizable spaces the definition of strong antishape would be entirely parallel to the (very tedious) 
definition of strong shape in \cite{Mard} and we refrain from spelling out the details.
Once this definition is written down, it should be possible to prove in the spirit of Theorem \ref{main1b}
that fine shape is stronger than strong antishape.
In fact, much of this predicted proof is already present in the proof of Theorem \ref{ash-surj}.

\begin{maintheorem} \label{main2b}
(a) For local compacta fine shape coincides with strong antishape (Theorem \ref{fish-sash}(a)).

(b) For coronated ANRs fine shape is strictly stronger than strong shape (Example \ref{comb-example}).
\end{maintheorem}

Part (a) says, in particular, that fine shape of local compacta admits a purely combinatorial description
(in terms of homotopy theory of countable simplicial complexes, without quantification over uncountable
sets of their subcomplexes).

Part (b) is very surprising, because it says that the dual of (a) fails.
The example is constructed using the fundamental group, and there is a result to the effect that no example 
of this kind can be constructed using homology (see Remark \ref{non-example}).
It is unknown if such examples can be constructed using higher homotopy groups.

\begin{problem} Does fine shape coincide with strong shape for locally connected coronated ANRs?
\end{problem}

\subsection{Fox shape versus Borsuk shape}

One effect of using a definition of high logical complexity, such as that of strong shape, is that it may be 
impossible to answer some of the simplest questions about it not only in practice, but also in ZFC.
It remains unknown whether Steenrod--Sitnikov homology of metrizable spaces is an invariant of strong shape. 
There is a modified version of Steenrod--Sitnikov homology, called {\it strong homology}, which is defined
in the spirit of strong shape and is an invariant of strong shape, but it cannot be computed in ZFC for
simplest non-compact non-ANRs (such as the product $\N\x\N^+$ of the countable discrete space and 
its one-point compactification) \cite{MP}.
In the words of Sklyarenko, the issue of strong shape invariance of Steenrod--Sitnikov homology 
(for paracompact spaces) ``will likely be a test not only for Steenrod--Sitnikov homology, but also for 
strong shape theory itself'' \cite{Sk6}.
The situation for \v Cech cohomology and antishape is entirely symmetric.
There is a modified version of \v Cech cohomology, called {\it strong cohomology} (see \cite{Mard}), which 
must be an invariant of strong antishape by design (once the definition of the latter is written out), 
but it cannot be computed in ZFC for simplest non-compact non-ANRs (see \cite{M-III}*{Example \ref{lim:MP'}}).

The set-theoretic troubles that strong (co)homology brings to light are of course implicit in the very concept of
strong (anti)shape, and stem directly from the set-theoretic treatment of the uncountable 
indexing sets of inverse systems (see \cite{M-IV} and \cite{M-III}*{Introduction}).
It was Fox who first opened the Pandora's box, but actually he was rather cautious about it himself, reverting in 
a later paper from arbitrary inverse systems to inverse sequences:
``In [\cite{Fox1}] I gave a somewhat more general definition (inverse system) but I now believe that that more 
general definition, although perfectly workable, may have been a tactical error on my part. (In topology I know 
of no convincing reason for the general definition, although I have recently heard that in algebraic geometry 
there may be.)'' \cite{Fox2}

Theorem \ref{3-defs} may be regarded as saying that a ``corrected'' form of strong shape miraculously 
coincides with a ``corrected'' form of strong antishape, and they further coincide with fine shape.
Here the ``correction'' of strong antishape would be by taking into account a natural topology on
the indexing set (i.e.\ the poset of compact subsets in the given space), namely, the topology of
the Hausdorff metric (compare \cite{M-IV}).
The ``correction'' of strong shape would be by taking into account an appropriate topology on 
the indexing set (i.e.\ a directed set of polyhedra or ANRs which approximate the given space),
though there is less clarity about the topology in this case (see, however, 
\cite{M-IV}*{proof of Theorem \ref{pos:lerayss2}}).
In both cases one may wonder whether the ``correction'' can be carried out in a more formal
manner, which would replace the homotopy colimits over posets that are implicit in the definitions of 
strong shape and strong antishape by homotopy colimits over pospaces 
(see \cite{M-IV}*{Theorem \ref{pos:atoms}}).

Much of the previous discussion is summarized by the following diagram.
\bigskip

\begin{tikzcd}[row sep=1.5em,column sep=1.5em]
& \text{homotopy} \dar[blue,Rightarrow] & \\
& \text{fine shape} \drar[blue,Rightarrow]{\text{topology on }\kappa}
\dar[blue,Rightarrow] \dlar[',blue,Rightarrow]{\text{topology on }\nu} & \\
\text{strong shape} \dar[',blue,Rightarrow]{\,\lim^p\text{over }\nu\,} & 
\text{UCOSS} \dlar[blue,Rightarrow] \drar[blue,Rightarrow] & 
\text{strong antishape} \dar[blue,Rightarrow]{\,\lim^p\text{over }\kappa} \\
\text{shape} \dar[',blue,Rightarrow]{\,\lim^1\text{over }\kappa\,} 
\ar[dd,magenta,Rightarrow,bend left] &  & \text{antishape = CG strong shape}
\dar[blue,Rightarrow]{\,\lim^1\text{over }\nu} \ar[dd,magenta,Rightarrow,bend right]\\
\text{shape based on CG homotopy} &  & \text{CG shape} \\
\text{\color{magenta}\v Cech cohomology} &  & \text{\color{magenta}Steenrod--Sitnikov homology}
\end{tikzcd}
\bigskip

It should be noted that long after uncountable inverse systems became popular in shape theory,
Borsuk ignored the new fashion and remained faithful to his (countable) fundamental sequences.
Of course, they are even more obviously inadequate, since the directed set of open covers of a non-compact
space almost never contains a countable cofinal sequence (see \cite{M-V}*{Theorem \ref{cor:uc}}).
But on the other hand, Borsuk's fundamental sequences $X\to Y$, where $X$ and $Y$ are closed subsets of 
absolute retracts $M$ and $N$, respectively, are somewhat reminiscent of $X$-$Y$-approaching maps
$M_{[0,\infty)}\to N_{[0,\infty)}$ between the mapping telescopes, where $\dots\subset M_1\subset M_0=M$
and $\dots\subset N_1\subset N_0=N$ are open neighborhoods of $X$ and $Y$ such that $\bigcap_i M_i=X$
and $\bigcap_i N_i=Y$ (compare \S\ref{approaching cat} below).
Indeed, when $X$ and $Y$ are compact, fundamental sequences $X\to Y$ are easily seen to be just an equivalent 
reformulation of $X$-$Y$-approaching maps $M_{[0,\infty)}\to N_{[0,\infty)}$; but in the non-compact case 
the two notions are not equivalent (by considering, for instance, \cite{Bor}*{example on p.~79}).

To summarize, the approach of the present paper can be said to reconcile
\begin{itemize}
\item Borsuk's and Fox's concepts of shape;
\item shape theories of types ($\kappa$), ($\nu$) and ($\kappa\nu$);
\item the adequacy of those shape theories that involve coherence conditions with the simplicity of those that don't.
\end{itemize}

\section{Fine shape} \label{fineshape}

\subsection{Filtered and cofiltered spaces} \label{filtrations}

A poset $P$ is called {\it directed} if any two elements of $P$ have an upper bound; and {\it codirected} 
if any two elements of $P$ have a lower bound.
(This is a special case of filtered and cofiltered categories, which we do not need.)

A space $W$ endowed with a (co)directed family of subsets $\kappa$ (resp.\ $\nu$), ordered by inclusion, 
will be called {\it (co)filtered} and denoted $W_\kappa$ (resp.\ $W^\nu$).

Let $X$ be a closed subset of a space $M$.
If $U$ and $V$ are open neighborhoods of $X$ in $M$, then so is $U\cap V$.
Thus 
\[\nu_M(X)\bydef \{U\but X\mid U\text{ is an open neighborhood of $X$ in $M$}\}\]
is a codirected family of open subsets of $M\but X$.
We will call it the {\it neighborhood cofiltration} and often abbreviate to $\nu X$.
It is easy to see that every open subset of $M\but X$ containing a member of $\nu X$ is itself a member of $\nu X$.

$\nu_M(X)$ contains the same information as the {\it co-neighborhood filtration} 
\[\bar\nu_M(X)\bydef \{F\subset M\but X\mid F\text{ is closed in $M$}\},\]
which is sometimes more convenient to use, and will often be abbreviated as $\bar\nu X$.
It is a directed family of closed subsets of $M\but X$.

On the other hand, if $K$ and $L$ are compact subsets of $M$, then so is $K\cup L$.
Thus
\[\kappa_M(X)\bydef \{K\but X\mid K\text{ is a compact subset of $M$}\}\]
is a directed family of closed subsets of $M\but X$.
We will call it the {\it compact filtration} and often abbreviate to $\kappa X$.
It is easy to see that every closed subset of $M\but X$, contained in a member $\kappa X$, is itself
a member of $\kappa X$.
Each member of $\kappa X$ is locally compact (since it is an open subset of a compact space).

\begin{example} (a) If $X=\emptyset$, then $\kappa X$ consists of all compact sets and $\bar\nu X$ 
of all closed sets.

(b) If $X=pt$ and $M$ is compact (i.e.\ it is the one-point compactification of $M\but X$), then 
$\kappa X$ consists of all closed sets and $\bar\nu X$ of all compact sets.
\end{example}

\begin{proposition} \label{duality}
Let $X$ be a closed subset of a metrizable space $M$ and let $S$ be a closed subset of $M\but X$.
Then

(a) $S\in\bar\nu X$ if and only if $S$ meets every member of $\kappa X$ in a compact set.

(b) $S\in\kappa X$ if and only if $S$ meets every member of $\bar\nu X$ in a compact set.
\end{proposition}

Thus $\kappa X$ and $\nu X$ in fact determine each other.

\begin{proof} First let us note that if $F$ is a closed subset of $M$, disjoint from $X$, and $K$ is 
a compact subset of $M$, then $F\cap K$ is a closed subset of $K$, and so must be compact.
Also it is disjoint from $X$, and so coincides with $F\cap (K\but X)$.
This implies the ``only if'' assertions in (a) and (b).

Suppose that $S$ is not a member of $\bar\nu X$.
Then $S$ is not closed in $M$.
Since $S$ is closed in $M\but X$, some sequence of points $x_i\in S$ converges to some $x\in X$.
Let $K=\{x_i\mid i\in\N\}$.
Then $K\cup\{x\}$ is compact, so $K\in\kappa X$.
Also $K$ is non-compact and $K\subset S$.
Thus $S$ meets the member $K$ of $\kappa X$ in the non-compact set $K$. 

Suppose that $S$ is not an element of $\kappa X$.
Then the closure $\bar S$ of $S$ in $M$ is non-compact.
Hence there exists a sequence of points $x_i\in\bar S$ which has no cluster points in $\bar S$.
Since $\bar S$ is closed in $M$, it also has no cluster points in $M$.
Each $x_i$ is a limit of a sequence of points $x_{ij}\in S$.
If $d$ is a metric on $M$, we may assume up to a renumbering that $d(x_{ij},x_i)<1/j$ for each $j$.
Then the sequence $(x_{ii})$ also has no cluster points.
Hence $F\bydef \{x_{ii}\mid i\in\N\}$ is non-compact and closed in $M$.
Also $F\subset S$, and therefore $F\in\bar\nu X$.
Thus $S$ meets the member $F$ of $\bar\nu X$ in the non-compact set $F$.
\end{proof}

If $\lambda$ is a directed family of closed subsets of a space $W$, let $\lambda^*$ denote the family of all
closed subsets of $W$ that meet every member of $\lambda$ in a compact set.
Clearly $\lambda^*$ is directed.
In these terms, Proposition \ref{duality} says that $\kappa X$ and $\bar\nu X$ are mutually dual in the sense that
$\kappa X=(\bar\nu X)^*$ and $\bar\nu X=(\kappa X)^*$.
It is easy to see that for any $\lambda$ as above, $\lambda^*$ and $\lambda^{**}$ are also mutually dual.%
\footnote{In fact, $*$ is an antitone Galois connection from the poset of all directed families of
closed subsets of $W$, ordered by inclusion, to the same poset.}

\begin{remark} 
In connection with Theorem \ref{homology-theories} one may wonder if the construction of $\kappa X$ and 
$\bar\nu X$ can be generalized to the setting where $X$ is itself endowed with 
a pair $\alpha=(\lambda,\mu)$ of directed families of its closed subsets such that 
$\lambda^*=\mu$ and $\mu^*=\lambda$.
If $X$ is of the form $X=N\but Y$, where $Y$ is a closed subset of $N$, and $\alpha=(\kappa Y,\bar\nu Y)$, 
then we want the generalized construction to be given by $\kappa X_\alpha=\kappa_{M\cup N} N$ and 
$\bar\nu X_\alpha=\bar\nu_{M\cup N} N$.
(The original construction of $\kappa X$ and $\bar\nu X$ would then correspond to the case $Y=\emptyset$.)
It turns out that such a generalization is impossible, as the following example shows.

Let $M=(-1,1)\x[0,1)$, $X=(-1,1)\x\{0\}$, $N=[-1,1)\x\{0\}$ and $Y=\{(-1,0)\}$.
Then $M\but X$ is homeomorphic to $W\bydef (-\infty,\infty)\x[0,\infty)$ by a homeomorphism $h$
given by $h(a,b)=(\tan\frac{a\pi}2,\tan\frac{b\pi}2)$.
Thus $\Gamma^+\bydef \{(x,\frac1x)\mid x\ge 1\}$ and $\Gamma^-=\{(-x,\frac1x)\mid x\ge 1\}$ are
quarter-hyperbolas lying in $W$.
Then we want $h^{-1}(\Gamma^+)$ to belong to $\bar\nu X_\alpha$ and not to belong to 
$\kappa X_\alpha$.
(Note that it belongs to $\bar\nu_{M\cup N}N$ and does not belong to $\kappa_{M\cup N}N$.)
We also want $h^{-1}(\Gamma^-)$ to belong to $\kappa X_\alpha$ and not to belong to 
$\bar\nu X_\alpha$.
(Note that it belongs to $\kappa_{M\cup N}N$ and does not belong to $\bar\nu_{M\cup N}N$.)
However, these two requirements are contradictory, since $h^{-1}(\Gamma^+)$ and $h^{-1}(\Gamma^-)$ are 
related by a self-homeomorphism of $M$ keeping $X$ fixed.
\end{remark}

\subsection{Approaching maps}

A continuous map $\phi\:A\to B$ is called {\it proper} if $\phi^{-1}(Q)$ is compact for every 
compact $Q\subset B$.

\begin{lemma}\label{mutually dual}
Let $W_1$ and $W_2$ be metrizable spaces and let $f\:W_1\to W_2$ be a continuous map.
Let $\kappa_i$ and $\bar\nu_i$ be directed families of closed subsets of $W_i$ 
such that $\kappa_i^*=\bar\nu_i$ and $\bar\nu_i^*=\kappa_i$.
Then the following are equivalent:
\begin{enumerate}
\item if $F\in\bar\nu_2$, then $f^{-1}(F)\in\bar\nu_1$;
\item if $K\in\kappa_1$, then $f(K)\in\kappa_2$ and $f|_K\:K\to f(K)$ is a proper map.
\end{enumerate}
\end{lemma}

\begin{proof}
{\it (2)$\Rightarrow$(1).}
Let $F\in\bar\nu_2$.
Given a $K\in\kappa_1$, by the hypothesis $f(K)\in\kappa_2$.
Then $Q\bydef f(K)\cap F$ is compact.
By the hypothesis $f|_K\:K\to f(K)$ is proper, hence $(f|_K)^{-1}(Q)$ is also compact.
It is easy to see that $(f|_K)^{-1}(Q)=K\cap f^{-1}(F)$.
Thus $f^{-1}(F)$ meets every $K\in\kappa_1$ in a compact set.
Hence $f^{-1}(F)\in\bar\nu_1$. 

{\it (1)$\Rightarrow$(2).} We will prove a slightly stronger assertion:
\begin{enumerate}
\item[(2$'$)] if $K\in\kappa_1$, then $f(K)\in\kappa_2$ and $f|_K\:K\to W_2$ is a proper map.
\end{enumerate}

Let $K\in\kappa_1$.
Given an $F\in\bar\nu_2$, by the hypothesis $f^{-1}(F)\in\bar\nu_1$.
Then $L\bydef K\cap f^{-1}(F)$ is compact.
Hence $f(L)$ is compact.
It is easy to see that $f(L)=f(K)\cap F$.
Thus $f(K)$ meets every $F\in\bar\nu_2$ in a compact set.
Hence $f(K)\in\kappa Y$.

Next we show that $f|_K\:K\to W_2$ is a proper map.
Let $Q$ be a compact subset of $W_2$.
Then $Q$ meets every member of $\kappa_2$ in a compact set, so $Q\in\bar\nu_2$.
Hence by the hypothesis $f^{-1}(Q)\in\bar\nu_1$.
Then $f^{-1}(Q)\cap K$ is compact.
But clearly $(f|_K)^{-1}(Q)=f^{-1}(Q)\cap K$.
\end{proof}

\begin{proposition} \label{map duality1}
Let $X$ and $Y$ be closed subsets of metrizable spaces $M$ and $N$, respectively, and let 
$f\:M\but X\to N\but Y$ be a continuous map.
The following are equivalent:
\begin{enumerate}
\item if $U\in\nu Y$, then $f^{-1}(U)\in\nu X$;
\item if $K\in\kappa X$, then $f(K)\in\kappa Y$ and $f|_K\:K\to f(K)$ is a proper map;
\item $f$ sends every sequence of points in $M\but X$ that has a cluster point in $X$ into
a sequence in $N\but Y$ that has a cluster point in $Y$.
\end{enumerate}
\end{proposition}

\begin{proof} Condition (1) can be equivalently reformulated as follows:
\begin{enumerate}
\item[(1$'$)] if $F\in\bar\nu Y$, then $f^{-1}(F)\in\bar\nu X$.
\end{enumerate}
By Lemma \ref{mutually dual} we have (1$'$)$\Leftrightarrow$(2).

{\it (3)$\Rightarrow$(1$'$).}
Suppose that $f^{-1}(F)\notin\bar\nu X$ for some $F\in\bar\nu Y$.
Then $f^{-1}(F)$ is closed in $M\but X$ but not in $M$.
Hence it contains a sequence $(x_n)$ converging to a point $x\in X$.
Since $F$ is closed in $N$, $\big(f(x_n)\big)$ has no cluster points in $Y$, 
which is a contradiction.

{\it (1$'$)$\Rightarrow$(3).} Suppose that $f$ sends some sequence of points 
$x_n\in M\but X$ that has a cluster point in $X$ into a sequence that has 
no cluster points in $Y$.
Let $F$ be the closure of $\{f(x_n)\mid n\in\N\}$ in $N$. 
Then $F$ lies in $N\but Y$.
Hence $f^{-1}(F)\in\bar\nu X$ by the hypothesis.
Since each $x_n\in f^{-1}(F)$, the sequence $(x_n)$ has no cluster points in $X$, 
which is a contradiction.
\end{proof}

If $X$ and $Y$ are closed subsets of metrizable spaces $M$ and $N$, respectively, a continuous map 
$f\:M\but X\to N\but Y$ will be called an {\it $X$-$Y$-approaching map} (or simply an {\it approaching map})
if it satisfies any of the equivalent conditions of Proposition \ref{map duality1}.

\begin{remark} Most revealing geometrically is perhaps the following characterization of approaching
maps, which is midway between conditions (2) and (3) of Proposition \ref{map duality1}.
A map $M\but X\to N\but Y$ is $X$-$Y$-approaching if and only if $f$ sends every sequence of points 
in $M\but X$ that converges to a point in $X$ into a sequence in $N\but Y$ whose underlying set
has compact closure in $M$.
(The proof is similar to that of Proposition \ref{map duality1}.)
\end{remark}

\begin{example} There exists a map of the form $f\:X\x(0,1]\to N\but Y$ whose restriction to $\{x\}\x(0,1]$ 
is $\{x\}$-$Y$-approaching for each $x\in X$, but $f$ is not $X$-$Y$-approaching.
The first example of this kind was constructed by V. Zemlyanoy.

Let $\N=\{1,2,\dots\}$ be the countable discrete space and $\N^+=\N\cup\{\infty\}$ its one-point compactification.
Let $f\:\N^+\x (0,1]\to [0,\infty)\x (0,1]$ be defined by
\[f(n,t)=
\begin{cases}
(0,t) &\text{ if }n=\infty\text{ or }t\ge\frac1n;\\
\big(n-n^2t,\,t)&\text{ if }n\ne\infty\text{ and }t<\frac1n.
\end{cases}
\]
For each $n\in\N^+$, the restriction $f|_{\{n\}\x(0,1]}$ extends to a continuous map
$\{n\}\x[0,1]\to [0,\infty)\x [0,1]$, and hence is $\{n\}$-$[0,\infty)$-approaching.
On the other hand, the sequence $x_n\bydef (n,\frac1{2n})$ converges to $(\infty,0)$ in $\N^+\x[0,1]$,
but the sequence $f(x_n)=(\frac n2,\frac1{2n})$ has no cluster points in $[0,\infty)\x[0,1]$.
Thus $f$ is not $\N^+$-$[0,\infty)$-approaching.
\end{example}

\subsection{Filtered and cofiltered maps}

A continuous map $f\:M\to N$ will be called a {\it filtered} map between filtered spaces $M_\kappa$
and $N_\lambda$ if every $K\in\kappa$ is sent by $f$ into some $L\in\lambda$; and a {\it cofiltered} map
between cofiltered spaces $M^\nu$ and $N^\mu$ if the $f$-preimage of every $U\in\mu$ contains 
some $V\in\nu$.

If we write $L=:f_\star(K)$ and $V=:f^\star(U)$, then $f_\star\:\kappa\to\lambda$ and $f^\star\:\mu\to\nu$
must be (not necessarily monotone) maps such that $f(K)\subset f_\star(K)$ for all $K\in\kappa$ 
and $f^\star(U)\subset f^{-1}(U)$ for all $U\in\mu$.
While these maps $f_\star$ and $f^\star$ are often convenient, they are not a part of the structure; 
only their existence is assumed.%
\footnote{Let us note that every cofiltered map $f\:M^\nu\to N^\mu$ induces a $\pro\Top$-morphism 
$\Xi^f$ between the inverse systems formed by the members of $\nu$ and $\mu$, and every choice of 
$f^\star$ yields an $\inv\Top$-morphism representing $\Xi^f$.
Similarly, every filtered map $f\:M_\kappa\to N_\lambda$ induces a $\ind\Top$-morphism $\Xi_f$ 
between the direct systems formed by the members of $\kappa$ and $\lambda$, and every choice of 
$f_\star$ yields a $\dir\Top$-morphism representing $\Xi_f$.
(See e.g.\ \cite{M00}*{\S\ref{book:pro}, \S\ref{book:ind}} for the definitions of $\pro\Top$,
$\ind\Top$, $\inv\Top$ and $\dir\Top$.)} 

Given filtered maps $M_\kappa\xr{f} M'_{\kappa'}\xr{g} M''_{\kappa''}$, clearly their composition
is filtered, with $(gf)_\star=g_\star f_\star$.
Dually, given cofiltered maps $M_1^{\nu_1}\xr{f} M_2^{\nu_2}\xr{g} M_3^{\nu_3}$, clearly their 
composition is cofiltered, with $(gf)^\star=f^\star g^\star$.

\begin{proposition} \label{map duality2}
Let $X$ and $Y$ be closed subsets of metrizable spaces $M$ and $N$, respectively, and let 
$f\:M\but X\to N\but Y$ be a continuous map.
The following are equivalent:
\begin{enumerate}
\item[(0)] $f$ is an $X$-$Y$-approaching map;
\item[(1$^\star$)] $f$ is a cofiltered map $(M\but X)^{\nu X}\to(N\but Y)^{\nu Y}$;
\item[(2$_\star$)] $f$ is a filtered map $(M\but X)_{\kappa X}\to(N\but Y)_{\kappa Y}$ and its restriction 
to each member of $\kappa X$ is a proper map into some member of $\kappa Y$.
\end{enumerate}
\end{proposition}

\begin{proof} In the notation of Proposition \ref{map duality1}, trivially
(1)$\Rightarrow$(1$^\star$) and (2)$\Rightarrow$(2$_\star$).

{\it (1$^\star$)$\Rightarrow$(1).}
Given a $U\in\nu Y$, let $V\in\nu X$ be given by the hypothesis.
Since $V\cup X$ is a neighborhood of each $x\in X$, so is $f^{-1}(U)\cup X$.
Since $U$ is open in $N\but Y$, $f^{-1}(U)$ is open in $M\but X$.
Hence $f^{-1}(U)\cup X$ is a neighborhood of each $x\in f^{-1}(U)$.
Thus $f^{-1}(U)\cup X$ is a neighborhood of every its point.
Hence $f^{-1}(U)\in\nu X$.

{\it (2$_\star$)$\Rightarrow$(2).}
Given a $K\in\kappa X$, let $L\in\kappa Y$ be given by the hypothesis.
We have $K=M_0\but X$ and $L=N_0\but Y$ for some compact $M_0\subset M$ and $N_0\subset N$.
Let $X_0=M_0\but K$ and $Y_0=N_0\but L$.
The members of $\bar\nu_{M_0} X_0$ and $\bar\nu_{N_0} Y_0$ are precisely the compact subsets 
of $K$ and $L$, respectively.
Since $f|_K\:K\to L$ is proper, we conclude that it satisfies (1$'$).
Hence by Proposition \ref{map duality1} it also satisfies (2).
In particular, $f|_K\:K\to f(K)$ is proper and $f(K)\in\kappa_{N_0} Y_0$.
Then clearly $f(K)\in\kappa Y$.
\end{proof}

\begin{remark}
So far we have seen only a very roundabout proof of (2$_\star$)$\Leftrightarrow$(3), 
including the entire proofs of Propositions \ref{duality} and \ref{map duality1} and 
Lemma \ref{mutually dual}.
But here is a direct proof.

{\it (2$_\star$)$\Rightarrow$(3).} 
Let $(x_n)$ be a sequence of points of $M\but X$ which has a cluster point $x\in X$.
Then some subsequence $(x_{n_k})$ converges to $x$.
Let $K=\{x\}\cup\{x_{n_k}\mid k\in\N\}$.
Since $K$ is compact, $f(K\but X)$ lies in a compact subset $L$ of $N$.
Then $\big(f(x_{n_k})\big)$ has a subsequence $(y_l)$ that converges to a point $y\in L$.
If $y\notin Y$, then $Q\bydef \{y\}\cup\{y_l\mid l\in\N\}$ is a compact subset of $L\but Y$ 
such that $\big(f|_{K\but X}\big)^{-1}(Q)$ is not closed in $M$ and hence not compact, 
which is a contradiction.
Hence $y\in Y$.

{\it (3)$\Rightarrow$(2$_\star$).} We will prove a slightly stronger assertion:
\begin{enumerate}
\item[(2$_\star'$)] if $K\in\kappa X$, then $f(K)$ lies in some $L\in\kappa Y$ and
$f|_K\:K\to N\but Y$ is proper.
\end{enumerate}

Let $K$ be a compact subset of $M$ and let $L$ be the closure 
of $f(K\but X)$ in $N$.
If $L$ is not compact, then it contains a sequence $(y_n)$ containing no convergent 
subsequences.
Let $y_n'\in L\but Y$ be such that $d(y_n,y'_n)\to 0$ as $n\to\infty$.
Then $(y'_n)$ also contains no convergent subsequences.
Each $y'_n=f(x_n)$ for some $x_n\in K\but X$.
Since $K$ is compact, $(x_n)$ has a cluster point $x\in K$.
If $x\in K\but X$, then $f(x)$ is a cluster point of $(y'_n)$, which is a contradiction.
If $x\in X$, then by the hypothesis $(y'_n)$ has a cluster point in $Y$, which is 
a contradiction.
Thus $L$ is compact.

Suppose that $N\but Y$ contains a compact subset $Q$ such that 
$Q^\star\bydef \big(f|_{K\but X}\big)^{-1}(Q)$ is not compact.
Then $Q^\star$ is closed in $M\but X$ but not closed in $K$.
Hence $Q^\star$ contains a sequence $(x_n)$ converging to a point $x\in X$.
Then $f(x_n)$ has a cluster point $y\in Y$.
Since each $f(x_n)\in Q$, we have $y\in Q$, contradicting $Q\cap Y=\emptyset$.
\end{remark}

\subsection{Approaching homotopies} 

A subset $S$ of a poset $P$ is called {\it cofinal} (resp.\ {\it coinitial}) in $P$ if every element of $P$ 
has an upper (resp.\ lower) bound in $S$.
Clearly, a cofinal (resp.\ coinitial) subset of a directed (resp.\ codirected) poset is also a directed 
(resp.\ codirected) poset.

It is easy to see that if $\kappa'$ is cofinal in $\kappa$ and $\lambda'$ is cofinal in $\lambda$,
a continuous map $f\:M\to N$ is a filtered map $M_\kappa\to N_\lambda$ if and only if it is a filtered map
$M_{\kappa'}\to N_{\lambda'}$.
Dually, if $\nu'$ is coinitial in $\nu$ and $\mu'$ is coinitial in $\mu$, then a continuous map $f\:M\to N$ is a 
cofiltered map $M^\nu\to N^\mu$ if and only if it is a cofiltered map $M^{\nu'}\to N^{\mu'}$.

\begin{lemma} \label{approaching homotopy}
Let $X$ and $Y$ be closed subsets of spaces $M$ and $N$, respectively, and let
$\sigma\otimes\tau=\{S\x T\mid S\in\sigma,\,T\in\tau\}$.
Then 

(a) $\kappa X\otimes\{I\}$ is cofinal in $\kappa(X\x I)$;

(b) $\nu X\otimes\{I\}$ is coinitial in $\nu(X\x I)$.
\end{lemma}

\begin{proof}
If $K\subset M\x I$ is compact, then so is $\pi(K)$, where $\pi\:M\x I\to M$ is the projection.
So $K$ is contained in the compact set $\pi(K)\x I$.
Thus $\kappa X\otimes\{I\}$ is cofinal in $\kappa(X\x I)$.

Let $U$ be an open neighborhood of $X\x I$ in $M\x I$.
Since $I$ is compact, $\pi\:M\x I\to M$ is a closed map (see \cite{M00}*{Lemma \ref{book:closed map}(a)}).
Hence $V\bydef M\but\pi(M\x I\but U)$ is open.
So $U$ contains the open neighborhood $V\x I$ of $X$.
Thus $\nu X\otimes\{I\}$ is coinitial in $\nu(X\x I)$.
\end{proof}

A homotopy $M\x I\to N$ between $f,g\:M\to N$ will be called
a {\it (co)filtered homotopy} between filtered maps $f,g\:M_\kappa\to N_\lambda$ 
(resp.\ cofiltered maps $f,g\:M^\nu\to N^\mu$) if it is a filtered map 
$(M\x I)_{\kappa\otimes\{I\}}\to N_\lambda$ (resp.\ a cofiltered map $(M\x I)^{\nu\otimes\{I\}}\to N^\mu$).

\begin{proposition} \label{homotopy duality}
Let $X$ and $Y$ be closed subsets of metrizable spaces $M$ and $N$, respectively, and 
let $H\:(M\but X)\x I\to N\but Y$ be a homotopy.
The following are equivalent:
\begin{enumerate}
\item $H$ is an $(X\x I)$-$Y$-approaching map;
\item $H$ is a cofiltered homotopy $(M\x I\but X\x I)^{\nu X\otimes\{I\}}\to (N\but Y)^{\nu Y}$;
\item $H$ is a filtered homotopy $(M\x I\but X\x I)_{\kappa X\otimes\{I\}}\to (N\but Y)_{\kappa Y}$ 
and its restriction to each member of $\kappa X\otimes\{I\}$ is a proper homotopy%
\footnote{That is, a homotopy that is a proper map.} in some member of $\kappa Y$.
\end{enumerate}
\end{proposition}

\begin{proof} The equivalence of (1) and (2) follows from Lemma \ref{approaching homotopy}(b).
The equivalence of (1) and (3) follows from Lemma \ref{approaching homotopy}(a) and the fact that
the restriction of a proper map to a closed subset of the domain is proper.
\end{proof}

If $X$ and $Y$ are closed subsets of metrizable spaces $M$ and $N$, respectively, a homotopy 
$H\:(M\but X)\x I\to N\but Y$ will be called an {\it $X$-$Y$-approaching homotopy} (or simply 
an {\it approaching homotopy}) if it satisfies any of the equivalent conditions of 
Proposition \ref{homotopy duality}. 
The set of $X$-$Y$-approaching homotopy classes of $X$-$Y$-approaching maps $M\but X\to N\but Y$
will be denoted $[X,Y]_{MN}$.

Given an $X$-$Y$-approaching map $f\:M\but X\to N\but Y$ and a $Y$-$Z$-approaching map
$g\:N\but Y\to L\but Z$, clearly their composition is an $X$-$Z$-approaching map.
Moreover, it is easy to see that the composition is well defined on approaching homotopy classes.

An $X$-$Y$-approaching map $f\:M\but X\to N\but X$ is an {\it $X$-$Y$-approaching homotopy equivalence} 
if there exists an $Y$-$X$-approaching map $g\:N\but Y\to M\but X$ such that $gf$ is 
$X$-$X$-approaching homotopic to $\id_{M\but X}$ and $fg$ is $Y$-$Y$-approaching homotopic 
to $\id_{N\but Y}$.

\subsection{Z-sets in absolute retracts} \label{z-sets}

By {\it absolute retracts} and {\it ANRs} we mean those for metrizable spaces.
See \cite{M00}*{\S\ref{book:anrs}} for their basic theory.

A subset $U$ of a space $M$ is called {\it homotopy dense} in $M$ if there exists 
a homotopy $h_t\:M\to M$ such that $h_0=\id$ and $h_t(M)\subset U$ for each $t>0$.
A subset of $X$ of a space $M$ is called {\it homotopy negligible} in $M$ if $M\but X$
is homotopy dense in $M$.

We call a subset $X$ of an ANR $M$ a {\it Z-set} in $M$ if $X$ is closed and homotopy
negligible in $M$.

\begin{lemma} \label{ar-embedding}
Every metrizable space $X$ embeds as a Z-set in some absolute retract $M$.
If $X$ is separable and/or completely metrizable, then so is $M$.
\end{lemma}

The constructions of $M$ in the proof and in the remarks below have the advantage of being 
quite elementary.
Later on we will discuss several other constructions, which have other advantages.

\begin{proof} 
By Wojdyslawski's theorem $X$ is homeomorphic to a closed subset of a convex subset $C$ of 
a normed vector space; if $X$ is separable and/or completely metrizable, then so is $C$ 
(see \cite{M00}*{Theorem \ref{book:wojdyslawski}}).
By Dugundji's theorem $C$ is an absolute retract (see \cite{M00}*{Theorem \ref{book:Dugundji extension}}).
Then $C\x I$ is also an absolute retract (see \cite{M00}*{Lemma \ref{book:ANR-basic}(a)}).
Also, $X\x\{1\}$ is clearly a Z-set in $C\x I$.
\end{proof}

\begin{remark} \label{Z-embeddings}
(a) Let us recall that $\R$, $[0,\infty)$ and $I=[0,1]$ are absolute retracts
(see \cite{M00}*{Theorem \ref{book:Dugundji extension}}) and that a finite or countable product of
absolute retracts is an absolute retract (see \cite{M00}*{Lemma \ref{book:ANR-basic}(a)}).

(b) If $X$ is a Polish space (=separable, completely metrizable), it embeds as a Z-set in the countable
product of lines $\R^\infty$ (see \cite{M00}*{Theorem \ref{book:Z-embedding}}). 

(c) If $X$ is a compactum (=compact metrizable space), it embeds in the Hilbert cube $I^\infty$
(see \cite{M00}*{Lemma \ref{book:product-embedding2}}), and hence embeds as a Z-set 
in $I^\infty\x I\cong I^\infty$.

(d) If $X$ is a local compactum (=locally compact separable metrizable space), it embeds
onto a closed subset of $I^\infty\x [0,\infty)$ (see \cite{M00}*{Proposition \ref{book:lc-emb}}).
Hence $X$ embeds as a Z-set in $I^\infty\x[0,\infty)\x I\cong I^\infty\x[0,\infty)$.

(e) If $X$ is a $n$-dimensional compactum, it embeds in $I^{2n+1}$, and hence embeds as a Z-set in $I^{2n+2}$.

(f) If $X$ is a $n$-dimensional local compactum, it embeds as a closed subset in $I^{2n}\x[0,\infty)$
(see \cite{M00}*{Corollary \ref{book:n-dim-emb}(b)}).
Hence $X$ embeds as a Z-set in $I^{2n+1}\x[0,\infty)$.
\end{remark}

\begin{lemma} \label{Milnor} 
Let $X$ be a closed subset of a metrizable space $M$, and let $Y$ be a Z-set in an absolute retract $N$.
Let $\phi\:X\to Y$ be a continuous map.
Then 

(a) $\phi$ extends to a continuous map $F\:M\to N$ such that $F^{-1}(Y)=X$;

(b) if $\phi$ is homotopic to a $\psi$ which extends to a $G\:M\to N$ 
such that $G^{-1}(Y)=X$, then $F$ is homotopic to $G$ by a homotopy $H_t$ such that 
$H_t^{-1}(Y)=X$ for each $t\in I$.
\end{lemma}

This is well-known at least for compacta (see \cite{M1}*{Lemma 2.1} and references there).

\begin{proof}[Proof. (a)] Since $N$ is an absolute retract (and hence an absolute extensor) and $X$ is closed, 
$\phi$ extends to a continuous map $f\:M\to N$.
Since $M$ is metrizable and $X$ is closed, there exists a continuous map $g\:M\to I$ such that $X=g^{-1}(0)$.
Since $Y$ is a Z-set in $N$, there exists a homotopy $h\:N\x I\to N$ such that $h_0=\id$ 
and $h_t(Y)\subset N\but Y$ for $t>0$.
Then $h^{-1}(Y)=Y\x\{0\}$, and hence the composition $F\:M\xr{f\x g}N\x I\xr{h}N$ is as desired.
\end{proof}

\begin{proof}[(b)] This follows from (a), using that $X\x I\cup M\x\partial I$ is a closed subset of $M\x I$ 
and $Y\x I\cup N\x\partial I$ is a Z-set in $N\x I$.
\end{proof}

\subsection{Fine shape category} \label{fsm-section}

Lemma \ref{Milnor} immediately implies the following

\begin{theorem} \label{fsm}
Suppose that $X$ and $Y$ are Z-sets in absolute retracts $M$ and $N$, respectively.
Then for every map $\phi\:X\to Y$ 

(a) there exists an $X$-$Y$-approaching map $\phi_{MN}\:M\but X\to N\but Y$ which combines with $\phi$ 
into a continuous map $M\to N$; 

(b) such a map $\phi_{MN}$ is unique up to $X$-$Y$-approaching homotopy;

(c) if $\phi$ is homotopic to $\psi$, then $\phi_{MN}$ is $X$-$Y$-approaching homotopic to $\psi_{MN}$.
\end{theorem}

Thus we get a map from the homotopy set $[X,Y]$ to the set $[X,Y]_{MN}$ of $X$-$Y$-approaching homotopy
classes of $X$-$Y$-approaching maps $M\but X\to N\but Y$.
Moreover, the map $[X,X]\to[X,X]_{MM}$ sends $[\id_X]$ to $[\id_{M\but X}]$ and the diagram
\[\begin{CD}
[X,Y]\x [Y,Z]@>\circ>>[X,Z]\\
@VVV@VVV\\
[X,Y]_{MN}\x [Y,Z]_{NL}@>\circ>>[X,Z]_{ML}
\end{CD}\]
commutes.

\begin{corollary} \label{fsm1}
Suppose that $X$ is a Z-set in two absolute retracts $M$ and $M'$.
Then $M\but X$ and $M'\but X$ are $X$-$X$-approaching homotopy equivalent. 
Specifically, the $X$-$X$-approaching maps $(\id_X)_{MM'}\:M\but X\to M'\but X$ and 
$(\id_X)_{M'M}\:M'\but X\to M\but X$ are mutually inverse up to $X$-$X$-approaching homotopy.
\end{corollary}

Suppose that $X$ is a Z-set in absolute retracts $M$ and $M'$, and $Y$ is a Z-set 
in absolute retracts $N$ and $N'$.
We say that $X$-$Y$-approaching maps $\Phi\:M\but X\to N\but Y$ and $\Phi'\:M'\but X\to N'\but Y$
represent the same {\it fine shape morphism} $X\to Y$ if the diagram
\[\begin{CD}
M\but X@>\Phi>>N\but Y\\
@V(\id_X)_{MM'}VV@VV(\id_Y)_{NN'}V\\
M'\but X@>\Phi'>>N'\but Y
\end{CD}\]
commutes up to $X$-$Y$-approaching homotopy. 
To be more precise, the relation that we have just defined on $X$-$Y$-approaching maps is 
a well-defined equivalence relation (it does not depend on the choice of the maps $(\id_X)_{MM'}$ 
and $(\id_Y)_{NN'}$ by Theorem \ref{fsm}(b) and it is symmetric by Corollary \ref{fsm1}), and 
fine shape morphisms $X\to Y$ are defined as its equivalence classes.

If $X$ and $Y$ are metrizable spaces, let $[X,Y]_\Fish$ denote the set of fine shape morphisms $X\to Y$.
Lemma \ref{ar-embedding} and Theorem \ref{fsm}(c) yield a map from the set $[X,Y]$ of homotopy classes 
of maps $X\to Y$ into $[X,Y]_\Fish$, given by $[\phi]\mapsto[\phi_{MN}]$.
It follows from Theorem \ref{fsm}(b) that this map does not depend on the choice of 
the absolute retracts $M$ and $N$.

\begin{proposition} \label{representation}
Let us fix embeddings of $X$ and $Y$ as Z-sets in absolute retracts $M$ and $N$.
Then there is a bijection between $[X,Y]_\Fish$ and the set $[X,Y]_{MN}$ of 
$X$-$Y$-approaching homotopy classes of $X$-$Y$-approaching maps $M\but X\to N\but Y$.
\end{proposition}

\begin{proof} By assigning to an approaching homotopy class the fine shape morphism which it
represents we get a map $f\:[X,Y]_{MN}\to[X,Y]_\Fish$.

If $X$-$Y$-approaching maps $\Phi,\Psi\:M\but X\to N\but Y$ represent the same fine shape 
morphism $X\to Y$, then they are $X$-$Y$-approaching homotopic (since $(\id_X)_{MM}$ is 
$X$-$X$-approaching homotopic to $\id_{M\but X}$, and similarly for $(\id_Y)_{NN}$).
Hence $f$ is injective.

Given a fine shape morphism $X\to Y$, it must be represented by an $X$-$Y$-approaching map 
$\Phi\:M'\but X\to N'\but Y$ for some absolute retracts $M'$ and $N'$ containing $X$ and $Y$
as Z-sets. 
Then by Corollary \ref{fsm1} it is also represented by the composition 
$M\but X\xr{(\id_X)_{MM'}}M'\but X\xr{\Phi}N'\but Y\xr{(\id_Y)_{N'N}}N\but Y$.
Hence $f$ is surjective.
\end{proof}

Using Proposition \ref{representation}, we define composition of fine shape morphisms $X\to Y\to Z$ 
by composing their representative approaching maps $M\but X\to N\but Y\to L\but Z$.
Using Theorem \ref{fsm}(b) it is easy to check that the result does not depend on the choice of 
the absolute retracts $M$, $N$ and $L$.
Thus we have defined the {\it fine shape category} and we have constructed a forgetful functor 
$\F^\Ho_\Fish$ from the homotopy category of metrizable spaces to the fine shape category.

A fine shape morphism $\phi\:X\to Y$ is called a {\it fine shape equivalence} if it is represented by 
an approaching homotopy equivalence; or equivalently if there exists a fine shape morphism
$\psi\:Y\to X$ such that $\psi\phi=[\id_X]$ and $\phi\psi=[\id_Y]$.
When such a $\phi$ exists, $X$ and $Y$ are said to be of the same {\it fine shape}.

\begin{remark}
The reader is invited to review the variations of the previous definitions (given, in particular, by 
the three equivalent definitions of approaching maps) by thinking about $[pt,W]_\Fish$, where 
$W\subset\R^2$ is the comb-and-flea set (see Example \ref{comb-example}).
\end{remark}

\section{First steps}

\subsection{Germs of approaching maps} \label{germs}

In this section we give an alternative definition of the fine shape category, with absolute retracts replaced by ANRs.

Let $X$ and $Y$ be closed subsets of metrizable spaces $M$ and $N$, respectively.
Given open neighborhoods $U$ and $V$ of $X$ in $M$ such that $V\subset U$, the restriction map $r_{UV}\:[X,Y]_{UN}\to [X,Y]_{VN}$ 
is defined by $F\mapsto F(\id_X)_{VU}$ (using the notation of Theorem \ref{fsm}), or equivalently by $[f]\mapsto [f|_{V\but X}]$.
Let $[X,Y]^\infty_{MN}$ be the direct limit \[\colim([X,Y]_{UN},\,r_{UV};\,\NN_X)\] in the category of sets, 
where $\NN_X$ is the directed poset of all open neighborhoods of $X$ in $M$ ordered by reverse inclusion.
(Clearly, $\NN_X$ is isomorphic to $\bar\nu_X$ as a poset.)

Thus elements of $[X,Y]^\infty_{MN}$ are represented by $X$-$Y$-approaching maps of the form 
$f\:U\but X\to N\but Y$, where $U$ is an open neighborhood of $X$ in $M$; another such map $g\:V\but X\to N\but Y$ represents 
the same element of $[X,Y]^\infty_{MN}$ if and only if there exists an open neighborhood $W$ of $X$ in 
$U\cap V$ such that $f|_{W\but X}$ is $X$-$Y$-approaching homotopic to $g|_{W\but X}$.
The class of the map $f$ in $[X,Y]^\infty_{MN}$ will be denoted by $[f]_\infty$.

\begin{lemma} \label{Milnor2} 
Let $X$ be a closed subset of a metrizable space $M$, and let $Y$ be a Z-set in an ANR $N$.
Let $\phi\:X\to Y$ be a continuous map.
Then 

(a) $\phi$ extends to a continuous map $F\:U\to N$ of some open neighborhood $U$ of $X$ in $N$ 
such that $F^{-1}(Y)=X$;

(b) for every other extension $\tilde F\:\tilde U\to N$ of $\phi$ to an open neighborhood $\tilde U$ 
of $X$ in $M$ such that $\tilde F^{-1}(Y)=X$ there exists an open neighborhood $V$ of $X$ in $M$ and 
a homotopy $H_t\:V\to N$ between $F|_V$ and $\tilde F|_V$ such that $H_t|_X=\phi$ and 
$H_t^{-1}(Y)=X$ for each $t\in I$.
\end{lemma}

\begin{proof} Similarly to Lemma \ref{Milnor}.
\end{proof}

\begin{lemma} \label{open subset}
If $M$ is an ANR, $X$ is a Z-set in $M$ and $U$ is an open neighborhood of $X$ in $M$,
then $U$ is an ANR and $X$ is a Z-set in $U$.
\end{lemma}

\begin{proof} See \cite{M00}*{Lemmas \ref{book:ANR-basic}(c) and \ref{book:z-set-rel}(a)}.
\end{proof}

\begin{proposition} \label{representation3}
Let $X$ and $Y$ be Z-sets in ANRs $M$ and $N$.
Then there is a bijection between $[X,Y]_\Fish$ and $[X,Y]^\infty_{MN}$.
\end{proposition}

\begin{proof} Similarly to Proposition \ref{representation}, using Lemmas \ref{Milnor2} and \ref{open subset}.
\end{proof}

If $Z$ is a closed subset of a metrizable space $L$, we can define the composition map
$[X,Y]^\infty_{MN}\x [Y,Z]^\infty_{NL}\to[X,Z]^\infty_{ML}$ by $[g]_\infty[f]_\infty=[g(f|_O)]_\infty$, where $O$ is 
the $f$-inverse of the domain of $g$.
It is not hard to see that this does not depend of the choice of the representatives, using that by 
Lemma \ref{approaching homotopy}(b) for every $X$-$Y$-approaching homotopy $f_t\:U\but X\to N\but Y$, 
where $U\in\NN_X$, and each $V\in\NN_Y$ there exists a $W\in\NN_X$ such that $W\but X\subset f_t^{-1}(V\but Y)$ for each $t\in I$.

\subsection{Homotopy}

\begin{theorem} \label{ANR}
If $Q$ is an ANR, then the forgetful map $\F^\Ho_\Fish\:[X,Q]\to[X,Q]_\Fish$ is a bijection 
for every metrizable space $X$.
\end{theorem}

This could be expected, but it takes a bit of work to prove it.

\begin{proof} 
We will use Proposition \ref{representation3}.
Let $N$ be an ANR containing $X$ as a Z-set, and let $M$ be an open neighborhood of $X$ in $N$.
Then there exists a homotopy $\phi_t^M\:N\to N$ such that $\phi_0=\id_N$, 
$\phi_t|_{N\but M}=\id_{N\but M}$ for all $t\in I$ and $\phi_t(N)\subset N\but X$ 
for all $t>0$ \cite{M00}*{Lemma \ref{book:z-set-rel}}.
Given another open neighborhood $L$ of $X$ in $N$, we have the homotopy 
$\Phi_t^{ML}\:N\to N\but X$ between $\phi_1^M$ and $\phi_1^L$ defined by 
$\Phi_t^{ML}(x)=\phi_{1-t}^M\big(\phi_t^L(x)\big)$.
Let us note that $\Phi_t^{ML}|_{N\but (M\cup L)}=\id_{M\but (M\cup L)}$ for all $t\in I$.

Given an $X$-$Q$-approaching map $G\:M\but X\to Q\x(0,1]$, let us consider 
the composition $g\:X\xr{\phi_1^M|_X}M\but X\xr{G}Q\x(0,1]\xr{\pi_1}Q$, where $\pi_1$
is the projection onto the first factor.

Given another $X$-$Q$-approaching map $G'\:M'\but X\to Q\x(0,1]$, where $M'$ is another 
open neighborhood of $X$ in $N$, and an $X$-$Q$-approaching homotopy $H_t\:L\but X\to Q\x(0,1]$ 
between $H_0=G|_{L\but X}$ and $H_1=G'|_{L\but X}$ for some open neighborhood $L$ of $X$ in $M\cap M'$,
the map $g$ is homotopic to the composition $g'\:X\xr{\phi_1^{M'}|_X}M'\but X\xr{G'}Q\x(0,1]\xr{\pi_1}Q$ 
via the stacked combination of the homotopies 
\begin{align*}
& X\xr{\Phi_t^{ML}|_X}M\but X\xr{G}Q\x(0,1]\xr{\pi}Q\\
& X\xr{\ \,\phi_1^L|_X\ \,}L\but X\xr{H_t}Q\x(0,1]\xr{\pi}Q\\
& X\xr{\Phi_t^{LM'}|_X}M'\but X\xr{G'}Q\x(0,1]\xr{\pi}Q.
\end{align*}
Thus we get a map $\sigma\:[X,Q]_\Fish\to[X,Q]$.

On the other hand, if $G$ happens to be the restriction of a continuous map 
$F\:M\to Q\x I$ whose another restriction is $f\:X\to Q\x\{0\}$, then $g$ 
is homotopic to $f$ via $h_t\:X\xr{\phi_t|_X}M\xr{F}Q\x I\xr{\pi}Q$.
Thus the composition $[X,Q]\xr{\F^\Ho_\Fish}[X,Q]_\Fish\xr{\sigma}[X,Q]$
is the identity, and hence $\F^\Ho_\Fish$ is injective.

Now let us consider the homotopy 
$\Phi_t\:M\but X\xr{\phi_t|}M\but X\xr{G}Q\x(0,1]\xr{\pi_1}Q$.
Its final map $\Phi_1$ extends via $g$ to a continuous map
$\bar\Phi_1\:M\to Q$.
On the other hand, the composition $\Psi\:M\but X\xr{G} Q\x(0,1]\xr{\pi_2}(0,1]$ 
with the projection onto the second factor extends to a continuous map 
$\bar\Psi\:M\to I$.
(Indeed, given a sequence of points $x_n\in M\but X$ converging to an $x\in X$,
the set $K\bydef \{x_1,x_2,\dots\}$ is a member of $\kappa X$, so $G$ sends it onto
$G(K)$ by a proper map.
Hence only finitely many of the $x_i$ land in $Q\x(\frac1n,1]$ for each $n$,
and consequently $\Psi(x_i)\to 0$.)
Thus $\Phi_t\x\Psi\:M\but X\to Q\x(0,1]$ is a homotopy from $\Phi_0\x\Psi=G$ 
to the map $\Phi_1\x\Psi$ which extends via $g\x\{0\}\:X\to Q\x\{0\}$ to 
a continuous map $\bar\Phi_1\x\bar\Psi\:M\to Q\x I$.

To prove that $\F^\Ho_\Fish\:[X,Q]\to[X,Q]_\Fish$ is surjective it remains to
show that $\Phi_t\x\Psi$ is an $X$-$Q$-approaching homotopy.
Given a compact set $K\subset M\x I$, the homotopy $\phi\:M\x I\to M$, 
$\phi(x,t)=\phi_t(x)$, sends $K$ onto a compact set $K'$.
Since $G$ is an $X$-$Q$-approaching map, $G(K'\but X)=K''\but Q\x\{0\}$ for some 
compact set $K''\subset Q\x I$.
Hence the composition $\Phi\:(M\but X)\x I\xr{\phi|}M\but X\xr{G}Q\x(0,1]\xr{\pi_1}Q$
sends $K\but X\x I$ into the compact set $L\bydef \pi_2(K'')$.
Consequently $\Phi\x(\Psi c)\:(M\but X)\x I\to Q\x (0,1]$, where 
$c\:(M\but X)\x I\to M\but X$ is the projection, sends $K\but X\x I$ into
$L\x (0,1]$, which is a member of $\kappa_Q$.
It remains to show that the restriction $r\:K\but X\x I\to L\x (0,1]$ of 
$\Phi\x(\Psi c)$ is a proper map.
Now $r^{-1}(L\x[\frac1n,1])=(K\but X\x I)\cap\Psi^{-1}([\frac1n,1])\x I$,
where $\Psi^{-1}([\frac1n,1])=G^{-1}(Q\x [\frac1n,1])$ is a member of $\bar\nu X$
since $G$ is an $X$-$Q$-approaching map and $Q\x [\frac1n,1]$ is a member of $\bar\nu Q$.
Since $K\but X\x I$ is a member of $\kappa(X\x I)$ and $\Psi^{-1}([\frac1n,1])\x I$
is a member of $\bar\nu(X\x I)$, their intersection $r^{-1}(L\x[\frac1n,1])$
is compact (see Proposition \ref{duality}).
It follows that $r$ is a proper map.
\end{proof}

\subsection{Approaching category}\label{approaching cat}

In this section we show that if $M$ and $N$ are arbitrary metrizable spaces (rather than absolute retracts) 
and $X$ and $Y$ are their homotopy negligible closed subsets, then every $X$-$Y$-approaching homotopy class
$M\but X\to N\but Y$ already determines a fine shape morphism $X\to Y$ (but generally not every fine shape 
morphisms $X\to Y$ can be obtained in this way, unless $N$ is an absolute retract).

\begin{lemma}\label{X-X-telescope}
Let $M$ be a metrizable space and $\dots\subset X_1\subset X_0=M$ its subsets such that
each $\Cl X_{i+1}\subset\Int X_i$ and $X\bydef \bigcap_i X_i$ is homotopy negligible in $M$.
Then $M\but X$ is $X$-$X$-approaching homotopy equivalent to the mapping telescope $X_{[0,\infty)}$.
\end{lemma}

\begin{proof}
Let $X_i^\circ=X_i\but X$.
By \cite{M00}*{Proposition \ref{book:lc-telescope2}} $M$ is homotopy equivalent to 
$X^\circ_{[0,\infty)}\cup X\x\{\infty\}$.
Moreover, the homotopy equivalence is by construction given by a deformation retraction
of $X^\circ_{[0,\infty)}\cup X\x\{\infty\}$ onto a copy $e(M)$ of $M$ which keeps $e(X)=X\x\{\infty\}$ 
fixed.
Hence $M\but X$ is $X$-$X$-approaching homotopy equivalent to $X^\circ_{[0,\infty)}$.

Since $X$ is homotopy negligible in $M$, clearly $X\x [0,\infty]$ is homotopy negligible in $L\bydef M\x[0,\infty]$,
and in particular $Y\bydef X\x [0,\infty)$ is homotopy negligible in $L$.
Clearly $X_{[0,\infty)}$ is a neighborhood of $Y$ in $L$, and hence so is $N\bydef X_{[0,\infty)}\cup X\x\{\infty\}$.
Then by \cite{M00}*{Lemma \ref{book:instant takeoff2}} $Y$ is homotopy negligible in $N$, i.e.\ there exists 
a homotopy $H_t\:N\to N$ such that $H_0=\id_N$ and $H_t(N)\subset N\but Y$ for $t>0$.
Moreover, by the construction of $H_t$ it restricts to a homotopy of $\Int N$ and keeps $N\but\Int N$ fixed.
In particular, $H_t$ keeps $X\x\{\infty\}$ fixed and restricts to a homotopy $h_t$ of $X_{[0,\infty)}$. 
Since $h_t(X_{[0,\infty)})\subset X^\circ_{[0,\infty)}$ for $t>0$, we get that $X_{[0,\infty)}$ is 
$X$-$X$-approaching homotopy equivalent to $X^\circ_{[0,\infty)}$.
\end{proof}

\begin{lemma} \label{graph-approximation} 
Let $M$ and $N$ be metric spaces, let $X$ be a subset $M$, and let $f\:M\to N$ be a map.
Then for each $\eps>0$ there exists an open neighborhood $U$ of $X$ in $M$ such that for each $p\in U$ 
there exists an $x\in X$ such that $d(p,x)<\eps$ and $d\big(f(p),f(x)\big)<\eps$.
\end{lemma}

\begin{proof} Let $\Gamma\:M\to M\x N$ be defined by $\Gamma(p)=\big(p,f(p)\big)$.
Let $V$ be the $\eps$-neighborhood of $\Gamma(X)$ in $M\x N$ in the $l_\infty$ product metric 
$d\big((x,y),(x',y')\big)=\max\big(d(x,x'),\,d(y,y')\big)$.
Then $U\bydef \Gamma^{-1}(V)$ satisfies the desired property.
\end{proof}

\begin{lemma} \label{approaching extension} 
Let $M$ and $N$ be metric spaces, let $\dots\subset X_1\subset X_0$ be closed subsets of $M$
and $\dots\subset N_1\subset N_0$ be ANR subsets of $N$.
Also let $X\subset\bigcap_i X_i$ and $Y\subset\bigcap_i N_i$.

Given an $X$-$Y$-approaching map $f\:X_{[0,\infty)}\to N_{[0,\infty)}$, there exist open 
neighborhoods $M_i$ of $X_i$ in $M$ such that $\dots\subset M_1\subset M_0$ and $f$ 
extends to an $X$-$Y$-approaching map $M_{[0,\infty)}\to N_{[0,\infty)}$.
\end{lemma}

\begin{proof} Since each $N_i$ is an ANR, there exist open neighborhoods $U_i$ of $X_i$ in $M$ 
such that $\dots\subset U_1\subset U_0$ and $f$ extends to a map $F\:U_{[0,\infty)}\to N_{[0,\infty)}$.
By Lemma \ref{graph-approximation} for each $i$ there exists an open neighborhood $M_i$ of $X_i$ 
in $U_i$ such that for each $p\in M_i$ there exists an $x\in X_i$ such that $d(p,x)<\frac1i$ and 
$d\big(f(p),f(x)\big)<\frac1i$.
We may assume that $\dots\subset M_1\subset M_0$.
Then for each sequence $(p_i)$ in $M_{[0,\infty)}$ such that $d(p_i,X)\to 0$ there exists 
a sequence $(x_i)$ in $X_{[0,\infty)}$ such that $d(p_i,x_i)\to 0$ and $d\big(F(p_i),f(x_i)\big)\to 0$.
If $(p_i)$ has a cluster point in $X$, then so does $(x_i)$; hence $\big(f(x_i)\big)$ has a cluster point 
in $Y$, and therefore so does $\big(F(p_i)\big)$.
Thus the restriction of $F$ to $M_{[0,\infty)}$ is $X$-$Y$-approaching.
\end{proof}

Let us consider the {\it approaching category} whose objects are pairs of the form $(M,X)$, where $M$ is 
a metrizable space and $X$ its closed homotopy negligible subset, and whose morphisms $(M,X)\to(N,Y)$ are
the elements of $[X,Y]_{MN}$ (that is, $X$-$Y$-approaching homotopy classes of $X$-$Y$-approaching maps 
$M\but X\to N\but Y$).

\begin{theorem} \label{negligible}
(a) There is a functor $\F^\Ap_\Fish$ from the approaching category to the fine shape category
which acts on objects by $(M,X)\mapsto X$.

(b) Moreover, the following diagram commutes:
\[\begin{CD}
[(M,X),\,(N,Y)]@>\text{\rm restriction}>>[X,Y]\\
@VVV@VVV\\
[X,Y]_{MN}@>\F^\Ap_\Fish>>[X,Y]_{\Fish}.
\end{CD}\]
\end{theorem}

\begin{proof} For each object $(M,X)$ of the approaching category we choose once and forever 
an absolute retract $U$ containing $M$ as a Z-set and a metric on $U$.
Our construction of $\F^\Ap_\Fish$ will use this choice; in fact $\F^\Ap_\Fish$
itself does not depend on this choice, but it takes some extra work to prove this 
(see Remark \ref{choices} below).

We need to show that for any objects $(M,X)$ and $(N,Y)$ of the approaching category 
the composition $[(M,X),\,(N,Y)]\to [X,Y]\to [X,Y]_\Fish$ factors through $[X,Y]_{MN}$, 
and moreover it does so in a way that respects composition of morphisms and the identity morphisms.

Thus we are given absolute retracts $U$ and $V$ containing $M$ and $N$ as Z-sets, and we are given
some metrics on $U$ and $V$.
For $i>0$ let $X_i$ and $Y_i$ be the closed $\frac1i$-neighborhoods of $X$ in $M$ and of $Y$ in $N$,
and let $X_0=X$ and $Y_0=Y$.
For $i>1$ let $V_i$ be the open $\frac1{i-1}$-neighborhood of $Y$ in $V$, and let $V_0=V_1=V$.
Thus each $V_i$ is an ANR (see \cite{M00}*{Lemma \ref{book:ANR-basic}(c)}) and 
$Y_i$ is its closed subset.

Given an $X$-$Y$-approaching map $M\but X\to N\but Y$, by Lemma \ref{X-X-telescope} it yields 
an $X$-$Y$-approaching map $f\:X_{[0,\infty)}\to Y_{[0,\infty)}$ which is well-defined up to
$X$-$Y$-approaching homotopy.
By Lemma \ref{approaching extension} there exist open neighborhoods $U^f_i$ of $X_i$ in $U$ 
such that $\dots\subset U^f_1\subset U^f_0$ and the composition 
$X_{[0,\infty)}\xr{f}Y_{[0,\infty)}\subset V_{[0,\infty)}$ extends to an $X$-$Y$-approaching map 
$F\:U^f_{[0,\infty)}\to V_{[0,\infty)}$.
Since $V_0$ is an absolute retract, we may assume that $U^f_0=U$.
Since $U$ is an ANR, so is each $U^f_i$.
Consequently $U^f_{[0,\infty]}$ and $V_{[0,\infty]}$ are ANRs containing $X$ and $Y$ as Z-sets 
(see \cite{M00}*{Corollary \ref{book:extended-ANR}}).
In fact they are absolute retracts, since they deformation retract onto $U$ and $V$, 
which are contractible (see \cite{M00}*{Corollary \ref{book:contractible ANR}}).
Then $F$ determines a fine shape morphism $X\to Y$.

Given another $X$-$Y$-approaching map $g\:X_{[0,\infty)}\to Y_{[0,\infty)}$ which is 
$X$-$Y$-approaching homotopic to $f$, we similarly obtain a mapping telescope 
$U^g_{[0,\infty)}$ and an $X$-$Y$-approaching map $G\:U^g_{[0,\infty)}\to V_{[0,\infty)}$ 
extending the composition $X_{[0,\infty)}\xr{g}Y_{[0,\infty)}\subset V_{[0,\infty)}$.
Moreover, the given $X$-$Y$-approaching homotopy $h_t\:X_{[0,\infty)}\to Y_{[0,\infty)}$
between $f$ and $g$ similarly yields 
a $U^h_{[0,\infty)}\subset U^f_{[0,\infty)}\cap U^g_{[0,\infty)}$ and 
an $X$-$Y$-approaching homotopy $H_t\:U^h_{[0,\infty)}\to V_{[0,\infty)}$ extending 
the composition $X_{[0,\infty)}\xr{h_t}Y_{[0,\infty)}\subset V_{[0,\infty)}$ and such
that $H_0=F|_{U^h_{[0,\infty)}}$ and $H_1=G|_{U^h_{[0,\infty)}}$.
Thus $H_0$ and $H_1$ represent the same fine shape morphism $X\to Y$.
On the other hand, the inclusion $U^h_{[0,\infty)}\to U^f_{[0,\infty)}$ extends to 
the inclusion $U^h_{[0,\infty]}\to U^f_{[0,\infty]}$ and hence represents 
the fine shape morphism $[\id_X]$.
Similarly, the inclusion $U^h_{[0,\infty)}\to U^g_{[0,\infty)}$ also represents $[\id_X]$.
Hence $F$ and $G$ represent the same fine shape morphism $X\to Y$.

Thus we have constructed a map $\F_{XY}\:[X,Y]_{MN}\to [X,Y]_{\Fish}$.
The previous argument with $f=g$ shows that it does not depend on the choice of 
the neighborhoods $U^f_i$.
It follows from the definitions that the diagram
\[\begin{CD}
[(M,X),\,(N,Y)]@>>>[X,Y]\\
@VVV@VVV\\
[X,Y]_{MN}@>\F_{XY}>>[X,Y]_{\Fish}.
\end{CD}\]
commutes.
This implies in particular that $\F_{XX}$ sends $[\id_{M\but X}]\in [X,X]_{MM}$
to $[\id_X]\in[X,X]_\Fish$.
It is clear from the construction of $\F_{XY}$ that the diagram
\[\begin{CD}
[X,Y]_{MN}\x [Y,Z]_{NL}@>\circ>>[X,Z]_{ML}\\
@VV\F_{XY}\x\F_{YZ}V@VV\F_{XZ}V\\
[X,Y]_{\Fish}\x [Y,Z]_{\Fish}@>\circ>>[X,Z]_{\Fish}
\end{CD}\]
commutes.
\end{proof}

\begin{remark} \label{choices}
Let us show that the construction of the map $[X,Y]_{MN}\to [X,Y]_{\Fish}$
in the proof of Theorem \ref{negligible} does not depend on the choices of $U$, $V$
and their metrics.

Indeed, suppose that we are given a $U'$ and a $V'$, resulting in $X_i'$, $Y_i'$, $V'_i$
and an $f'\:X'_{[0,\infty)}\to Y'_{[0,\infty)}$, which in turn produces some $U^{f'}$
an an $F'\:U^{f'}_{[0,\infty)}\to V'_{[0,\infty)}$.
It follows from Lemma \ref{X-X-telescope} that there exist an $X$-$X$-approaching homotopy 
equivalence $\phi\:X'_{[0,\infty)}\to X_{[0,\infty)}$ and a $Y$-$Y$-approaching homotopy 
equivalence $\psi\:Y'_{[0,\infty)}\to Y_{[0,\infty)}$ such that the diagram
\[\begin{CD}
X'_{[0,\infty)}@>f'>>Y'_{[0,\infty)}\\
@V\phi VV@V\psi VV\\
X_{[0,\infty)}@>f>>Y_{[0,\infty)}
\end{CD}\tag{$*$}\]
commutes up to $X$-$Y$-approaching homotopy.
By the proof of Lemma \ref{X-X-telescope} $\phi$ and $\psi$ extend by $\id_X$ and $\id_Y$
to maps $\bar\phi\:X'_{[0,\infty]}\to X_{[0,\infty]}$ and 
$\bar\psi\:Y'_{[0,\infty]}\to Y_{[0,\infty]}$.

By Lemma \ref{approaching extension} there exist open neighborhoods $U^\phi_i$ of $X'_i$ 
in $U'$ such that $\dots\subset U^\phi_1\subset U^\phi_0$ and the composition 
$X'_{[0,\infty)}\xr{\phi}X_{[0,\infty)}\subset U^f_{[0,\infty)}$ extends to 
an $X$-$X$-approaching map $\Phi\:U^\phi_{[0,\infty)}\to U^f_{[0,\infty)}$.
Since $U^f_0=U$ is an absolute retract, we may assume that $U^\phi_0=U'$.
Since $U'$ is an ANR, so is each $U^\phi_i$.
Moreover, it is clear from the proof of Lemma \ref{approaching extension} that
the map $\Phi\cup\bar\phi\:U^\phi_{[0,\infty]}\to U^f_{[0,\infty]}$ is continuous.
Hence $\Phi$ represents the fine shape morphism $[\id_X]$.
Similarly we obtain a $U^\psi_{[0,\infty)}$ and a $Y$-$Y$-approaching map 
$\Psi\:U^\psi_{[0,\infty)}\to V_{[0,\infty)}$ representing the fine shape morphism $[\id_Y]$
and extending the composition $Y'_{[0,\infty)}\xr{\psi}Y_{[0,\infty)}\subset V_{[0,\infty)}$.

In the same fashion we get a $U^{\psi f'}_{[0,\infty)}$ and 
an $X$-$Y$-approaching map $\tilde F'\:U^{\psi f'}_{[0,\infty)}\to U^\psi_{[0,\infty)}$ 
extending the composition
$Y'_{[0,\infty)}\xr{f'}X'_{[0,\infty)}\subset U^\psi_{[0,\infty)}$.
We may assume that $U^{\psi f'}_{[0,\infty)}\subset U^{f'}_{[0,\infty)}$ and 
by inspecting the proof of Lemma \ref{approaching extension} we may choose $\tilde F'$ 
to be the restriction of $F'$.
Since the inclusions $U^{\psi f'}_{[0,\infty)}\subset U^{f'}_{[0,\infty)}$ and
$U^\psi_{[0,\infty)}\subset V'_{[0,\infty)}$ extend to inclusions
$U^{\psi f'}_{[0,\infty]}\subset U^{f'}_{[0,\infty]}$ and
$U^\psi_{[0,\infty]}\subset V'_{[0,\infty]}$, they represent the fine shape morphisms
$[\id_X]$ and $[\id_Y]$, respectively.
Hence $F'$ and $\tilde F'$ represent the same fine shape morphism $X\to Y$.

On the other hand, since the two compositions in ($*$) are $X$-$Y$-approaching homotopic, 
our previous argument shows that the compositions 
$U^{\psi f'}_{[0,\infty)}\xr{\tilde F'} U^\psi_{[0,\infty)}\xr{\Psi} V_{[0,\infty)}$
and $U^\phi_{[0,\infty)}\xr{\Phi} U^f_{[0,\infty)}\xr{F}V_{[0,\infty)}$
represent the same fine shape morphism.
Since $\Phi$ represents $[\id_X]$ and $\Psi$ represents $[\id_Y]$, we conclude that
$F$ represents the same fine shape morphism $X\to Y$ as $\tilde F'$, and hence also as $F'$.
\end{remark}

Theorem \ref{negligible}(a) immediately implies

\begin{corollary} \label{negligible2}
Let $M$ and $N$ be metrizable spaces and let $X$ and $Y$ be their closed 
homotopy negligible subsets.
If $M\but X$ and $N\but Y$ are $X$-$Y$-approaching homotopy equivalent, then $X$ and $Y$ 
are fine shape equivalent.
\end{corollary}

Let us consider the {\it germ-approaching category} whose objects are pairs of the form $(M,X)$, where $M$ is 
a metrizable space and $X$ its closed homotopy negligible subset, and whose morphisms $(M,X)\to(N,Y)$ are
the elements of $[X,Y]^\infty_{MN}$ (see \S\ref{germs}).

\begin{corollary}
There is a functor $\F^\Gap_\Fish$ from the germ-approaching category to the fine shape category,
which acts on objects by $(M,X)\mapsto X$.
\end{corollary}

\begin{proof}
By Theorem \ref{negligible}(b) $\F^\Ap_\Fish$ sends $[(\id_X)_{VU}]\in [X,X]_{UV}$ into 
$[\id_X]\in[X,X]_\Fish$, and hence commutes with the restriction map 
$r_{UV}\:[X,Y]_{UN}\to [X,Y]_{VN}$.
Then the universal property of colimit yields a map $\F_{XY}\:[X,Y]^\infty_{MN}\to[X,Y]_\Fish$.
It is easy to check that these maps $\F_{XY}$ form a functor.
\end{proof}

By elaborating on the proof of Theorem \ref{negligible} we also obtain the following more precise version of 
Propositions \ref{representation} and \ref{representation3}.

\begin{theorem} \label{representation2}
Let $M$, $N$ be metrizable spaces and $X$, $Y$ their closed homotopy negligible subsets.

(a) If $N$ is an absolute retract, then $\F^\Ap_\Fish\:[X,Y]_{MN}\to[X,Y]_\Fish$ is a bijection.

(b) If $N$ is an ANR, then $\F^\Gap_\Fish\:[X,Y]_{MN}^\infty\to[X,Y]_\Fish$ is a bijection.
\end{theorem}

\begin{proof}[Proof. (a)] We will assume familiarity with the construction of $\F^\Ap_\Fish$ in 
the proof of Theorem \ref{negligible}.
By Remark \ref{choices} we may assume that $V=N\x I$ with the $l_\infty$ product metric, 
where $N$ is identified with $N\x\{0\}$.
The projection $V=N\x I\to N$ sends each $V_i$, $i\ge 1$, into $Y_{i-1}$, and $V_0$ into $Y_0$.
This yields a map $V_{[0,\infty]}\to Y_{[0,\infty]}$ keeping $X$ fixed, which is easily seen
to be homotopy inverse, keeping $X$ fixed, to the inclusion $Y_{[0,\infty]}\to V_{[0,\infty]}$.
Hence the inclusion $j\:Y_{[0,\infty)}\to V_{[0,\infty)}$ is a $Y$-$Y$-approaching homotopy equivalence.
Therefore the map $j_*\:[X,Y]_{X_{[0,\infty]},Y_{[0,\infty]}}\to[X,Y]_{X_{[0,\infty]},V_{[0,\infty]}}$
is a bijection.

On the other hand, for $i>1$ let $U_i$ be the open $\frac1{i-1}$-neighborhood of $X$ in $U$, and 
let $U_0=U_1=U$.
Then $U_{[0,\infty]}$ is an absolute retract containing $X$ as a Z-set (by repeating the proof that 
$V_{[0,\infty]}$ is an absolute retract containing $Y$ as a Z-set) and also each $X_i$ is a closed subset
of $U_i$.
The inclusion $i\:X_{[0,\infty)}\to U_{[0,\infty)}$ induces a map
$i^*\:[X,Y]_{U_{[0,\infty)},V_{[0,\infty)}}\to[X,Y]_{X_{[0,\infty)},V_{[0,\infty)}}$.
By arguing as in the proof of Theorem \ref{negligible} it is not hard to show that the composition
\[[X,Y]_\Fish\xr{\simeq}[X,Y]_{U_{[0,\infty]},V_{[0,\infty]}}\xr{i^*}
[X,Y]_{X_{[0,\infty]},V_{[0,\infty]}}\xr{(j_*)^{-1}}[X,Y]_{X_{[0,\infty]},Y_{[0,\infty]}}\xr{\simeq} [X,Y]_{MN}\]
is inverse to $\F^\Ap_\Fish$.
\end{proof}

\begin{proof}[(b)] Along the lines of the proof of (a).
\end{proof}

\subsection{The graph of an approaching map}

Let $X$ and $Y$ be closed subsets of metrizable spaces $M$ and $N$ and let $f\:M\but X\to N\but Y$
be an $X$-$Y$-approaching map.
Let $\bar\Gamma_f$ be the closure of the graph 
$\Gamma_f\bydef \{(x,y)\in(M\but X)\x(N\but Y)\mid y=f(x)\}$ in $M\x N$.

\begin{lemma} \label{graph}
$\bar\Gamma_f\but\Gamma_f\subset X\x Y$.
\end{lemma}

\begin{proof}
Given an $(x,y)\in\bar\Gamma_f$, it is the limit of a sequence of pairs 
$\big(x_i,f(x_i)\big)\in\Gamma_f$.
Then $x=\lim x_i$ and $y=\lim f(x_i)$.
If $x\in M\but X$, then $y=f(x)$ since $f$ is continuous.
Suppose that $x\in X$.
Since $f$ is $X$-$Y$-approaching, $\big(f(x_i)\big)$ has a cluster point in $Y$, and therefore $y\in Y$.
\end{proof}

\begin{lemma} \label{perfect}
The projection $\phi\:\bar\Gamma_f\to M$ is a perfect (in particular, closed) map.
\end{lemma}

\begin{proof}
By \cite{Sak2}*{Proposition 2.1.6} to show that $\phi$ is a perfect map it suffices
to show that any sequence of pairs $(x_i,y_i)\in\bar\Gamma_f$ such that $(x_i)$ is convergent 
has a convergent subsequence.
Thus let $(x_i,y_i)\in\bar\Gamma_f$ be such that $(x_i)$ converges to some $x\in M$.
For each $i$ let $\big(x_{ij},f(x_{ij})\big)\in\Gamma_f$ be a sequence converging to $(x_i,y_i)$.
Let us fix some metrics on $M$ and $N$, and let $j_i$ be such that 
$d(x_i,x_{ij_i})<\frac1i$ and $d\big(y_i,f(x_{ij_i})\big)<\frac1i$.
Then $x_{ij_i}\to x$.
Since $f$ is $X$-$Y$-approaching, $\big(f(x_{ij_i})\big)$ has a cluster point $y$.
Then it has a subsequence converging to $y$.
Then $(y_i)$ also has a subsequence $(y_{i_k})$ converging to $y$.
Then $(x_{i_k},y_{i_k})\to (x,y)$.
Thus $\phi$ is perfect.
\end{proof}

\begin{remark} \label{relations} (Not used in the sequel.)

(a) By Lemma \ref{perfect} the projection $\phi\:\bar\Gamma_f\to M$ is a closed map, 
so the multi-valued map $\phi^{-1}\:M\To\bar\Gamma_f$ is continuous.%
\footnote{If $A$ and $B$ are spaces, a multi-valued map $F\:A\To B$ is called {\it continuous} (or upper semi-continuous) 
if the preimage $F^{-1}(C)=\{a\in A\mid F(a)\cap C\ne\emptyset\}$ of every closed subset $C$ of $B$ is closed.
This is equivalent to saying that given an $a\in A$ and a neighborhood $U$ of $F(a)$,
there exists a neighborhood $V$ of $a$ such that $F(V)\subset U$.}
Then its composition with the projection $\psi\:\bar\Gamma_f\to N$ is also continuous.
Thus we obtain a continuous multi-valued map $\bar f\:M\To N$.
By Lemma \ref{graph} it restricts to a multi-valued map $X\To Y$.

(b) Moreover, the composition of $\phi^{-1}$ with the inclusion map $\bar\Gamma_f\to M\x N$ is also continuous.
It follows that $\bar f$ is strongly continuous.%
\footnote{A multi-valued map $F\:A\To B$ is called {\it strongly continuous} if the map $\gamma_F\:A\to A\x B$, 
defined by $\gamma_F(a)=\{a\}\x F(a)$, is continuous.
According to \cite{Can}*{Theorem A.3}, a multi-valued map $F\:A\To B$ is strongly continuous if and only if whenever
a sequence of points $a_i\in A$ converges to an $a\in A$, and each $b_i\in F(a_i)$, then the sequence $(b_i)$
has a cluster point in the set $F(a)$.
In particular, if $F$ has compact images of points, then it is strongly continuous.}

(c) Still more can be said. 
By Lemma \ref{perfect} $\phi$ is perfect, so $\phi^{-1}\:M\To\bar\Gamma_f$ has compact point images. 
Therefore so does $\bar f$.
One consequence of this is that $\bar f(K)$ is compact for every compact $K\subset M$ \cite{Mich}*{Corollary 9.6}.
\end{remark}

\subsection{Map excision: sufficiency} \label{mex-section}

Let us recall that a (co)homology theory $H$ on the category of closed pairs of metrizable spaces 
satisfies the {\it map excision axiom} if every closed map $f\:(X,A)\to(Y,B)$ between such pairs 
that restricts to a homeomorphism between $X\but A$ and $Y\but B$, induces for each $n$ 
an isomorphism $H_n(X,A)\simeq H_n(Y,B)$ (respectively, $H^n(X,A)\simeq H^n(Y,B)$).

We say that $H$ is {\it fine shape invariant} if its restriction to single spaces 
(i.e.\ pairs of the form $(X,\emptyset)$), regarded as a covariant (contravariant) functor 
from the homotopy category of metrizable spaces to the category of graded abelian groups, 
factors through the fine shape category.

\begin{theorem} \label{mex-fish} 
Let $H$ be a (co)homology theory on closed pairs of metrizable spaces.
If $H$ satisfies the map excision axiom, then it is fine shape invariant.
\end{theorem}

\begin{proof}
We consider only the case of homology, the case of cohomology being similar.
Let $X$ and $Y$ be metrizable spaces, contained as $Z$-sets in absolute retracts $M$ and $N$,
and let a fine shape morphism $X\to Y$ be represented by an $X$-$Y$-approaching map 
$f\:M\but X\to N\but Y$.
Let $\Gamma_f$ be the graph of $f$ and let $\bar\Gamma_f$ be its closure in $M\x N$.
By Lemma \ref{graph} $V_f\bydef \bar\Gamma_f\but\Gamma_f$ lies in $X\x Y$.
The projections of $M\x N$ onto its factors restrict to maps $\phi\:(\bar\Gamma_f,V_f)\to(M,X)$
and $\psi\:(\bar\Gamma_f,V_f)\to(N,Y)$.
By Lemma \ref{perfect} $\phi$ is closed.

Let us pick some $x\in M\but X$, and let us write $X^+=X\cup\{x\}$, $Y^+=Y\cup\{f(x)\}$ 
and $V^+=V_f\cup\{\big(x,f(x)\big)\}$.
By the map excision axiom $\phi_*\:H_n(\bar\Gamma_f,V_f^+)\to H_n(M,X^+)$ 
is an isomorphism for each $n$.
Since $H_i(M,x)=0$ for all $i$, we have $H_n(M,X^+)\simeq H_{n-1}(X^+,x)\simeq H_{n-1}(X)$;
and similarly $H_n(N,Y_+)\simeq H_{n-1}(Y)$.
Hence the composition
$H_n(M,X^+)\xr{\phi_*^{-1}} H_n(\bar\Gamma_f,V_f^+)\xr{\psi_*}H_n(N,Y^+)$
yields a map $f_*\:H_{n-1}(X)\to H_{n-1}(Y)$.

Let $H\:(M\but X)\x I\to N\but Y$ be an $X$-$Y$-approaching homotopy between $f$ and $f'$.
Since $H$ is an $X\x I$-$Y$-approaching map, by repeating the previous arguments we obtain
a closed map $\Xi\:(\bar\Gamma_H,V_H)\to (M,\,X\x I)$.
Let $V_H^+=V_H\cup\{\big((x,t),H(x,t)\big)\mid t\in I\}$.
By the map excision axiom $\Xi_*\:H_n(\bar\Gamma_H,V_H^+)\to H_n(M\x I,\,X^+\x I)$ 
is an isomorphism.
Since the homomorphism $\iota_i\:H_n(M,X^+)\to H_n(M\x I,\,X^+\x I)$ induced by the inclusion
$\iota_i\:M=M\x\{i\}\subset M\x I$ is an isomorphism for $i=0,1$, we get that the
inclusion induced homomorphisms $H_n(\bar\Gamma_f,V_f^+)\to H_n(\bar\Gamma_H,V_H^+)$
and $H_n(\bar\Gamma_{f'},V_{f'}^+)\to H_n(\bar\Gamma_H,V_H^+)$ are isomorphisms.
Then it follows that $f_*=f'_*$.
Thus we may denote $f_*$ by $[f]_*$, where $[f]\:X\to Y$ is the fine shape morphism
represented by $f$.

If the $X$-$Y$-approaching map $f$ has a continuous extension $\bar f\:M\to N$, then it is easy
to see that $\bar\Gamma_f=\Gamma_{\bar f}$. 
It follows that $[f]_*=F_*$, where $F\:X\to Y$ is the restriction of $\bar f$.

It remains to show that what we have defined is a functor (from the fine shape category to
the category of graded abelian groups).
It is easy to see that $[\id]_*=\id$. 
Given a $Y$-$Z$-approaching map $g\:N\but Y\to L\but Z$, let us show that $[gf]_*=[g]_*[f]_*$.
Let $\Delta_{f,g}=\{(x,y,z)\in(M\but X)\x(N\but Y)\x(L\but Z)\mid y=f(x),\,z=g(y)\}$ and let 
$\bar\Delta_{f,g}$ be the closure of $\Delta_{f,g}$ in $M\x N\x L$.
Let $W_{f,g}=\bar\Delta_{f,g}\but\Delta_{fg}$ and $W_{f,g}^+=W_{fg}\cup\{\big(x,f(x),gf(x)\big)\}$.
Then the commutative diagram on the left induces the commutative diagram on the right
\[\begin{tikzcd}[row sep=0.5em,column sep=0em,cramped]
X & & X\x Z \ar[ll] \ar[rd] & \\
 & & & Z \\
X\x Y \ar[uu] \ar[rd] & & X\x Y\x Z 
\ar[uu] \ar[ll] \ar[rd] & \\
 & Y & & Y\x Z \ar[uu] \ar[ll]
\end{tikzcd}
\hspace{20pt}
\begin{tikzcd}[row sep=0.5em,column sep=0em,cramped]
H_n(M,X^+) & & H_n(\bar\Gamma_{gf},V_{gf}^+) \ar[ll, "\simeq"'] \ar[rd] & \\
 & & & H_n(L,Z^+) \\
H_n(\bar\Gamma_f,V_f^+) \ar[uu, "\simeq"] \ar[rd] & & H_n(\bar\Delta_{f,g},W_{f,g}^+) 
\ar[uu, "\simeq"] \ar[ll, "\simeq"'] \ar[rd] & \\
 & H_n(N,Y^+) & & H_n(\bar\Gamma_g,V_g^+) \ar[uu] \ar[ll, "\simeq"']
\end{tikzcd}\]
and the assertion follows.
\end{proof}

As a byproduct of the proof of Theorem \ref{mex-fish} we have

\begin{theorem} \label{mex-supports}
Let $H$ be a (co)homology theory on closed pairs of metrizable spaces satisfying 
the map excision axiom.
Suppose that $(X,A)$ and $(Y,B)$ are closed pairs of metrizable spaces and there exists 
a homeomorphism $h\:X\but A\to Y\but B$ such that $h$ is an $A$-$B$-approaching map
and $h^{-1}$ is a $B$-$A$-approaching map.
Then $H_*(X,A)\simeq H_*(Y,B)$ (resp.\ $H^*(X,A)\simeq H^*(Y,B)$).
\end{theorem}

\begin{proof}
Let $\bar\Gamma_h$ be the closure of the graph 
$\Gamma_h\bydef \{(x,y)\in(X\but A)\x(Y\but B)\mid y=h(x)\}$ in $X\x Y$.
Then by the proof of Theorem \ref{mex-fish}
$C\bydef \bar\Gamma_h\but\Gamma_h$ lies in $A\x B$ and the restrictions
$\phi\:\bar\Gamma_h\to X$ and $\psi\:\bar\Gamma_h\to Y$ of the projections of $X\x Y$ 
onto its factors are perfect maps.
Hence $\phi_*\:H_n(\bar\Gamma_h,C)\to H_n(X,A)$ and $\psi_*\:H_n(\bar\Gamma_h,C)\to H_n(Y,B)$
are isomorphisms for each $n$ by the map excision axiom, and similarly for cohomology.
\end{proof}

\subsection{Map excision: necessity}

Now we turn to the converse of Theorem \ref{mex-fish}.

For a map $f\:X\to Y$ we denote by $MC(f)$ its metric mapping cylinder (see \cite{M00}*{\S\ref{book:mmc}}).
For a subset $Z\subset Y$, we denote by $f_Z$ the restriction $f^{-1}(Z)\to Z$ of $f$ (which should not be 
confused with $f|_{f^{-1}(Z)}\:f^{-1}(Z)\to Y$ as they have different metric mapping cylinders).

\begin{lemma} \label{fish-MC} Let $f\:X\to Y$ be a perfect map between metrizable spaces
and let $A\subset X$ and $B\subset Y$ be closed subsets such that $f(A)\subset B$ and 
$f$ restricts to a homeomorphism $X\but A\to Y\but B$.
Then the composition $\phi\:X\cup MC(f_B)\subset MC(f)\xr{\pi} Y$ is a fine shape equivalence.
\end{lemma}

Let us note that since $f$ is perfect, the topology of $MC(f)$ coincides with the quotient topology
(see \cite{M00}*{Remark \ref{book:perfect mc}}).

\begin{proof} Let us write $\hat X=X\cup MC(f_B)$.
For $y\in Y\but B$ let $y^*\in X\but A$ denote the unique element of $f^{-1}(y)$.
Let us note that $\phi$ has a set-theoretic section, continuous as a map
$s\:B\sqcup (Y\but B)\to\hat X$, which is defined by $s(y)=y$ if $y\in B$ and 
by $s(y)=y^*$ if $y\in Y\but B$.
Although $s$ is discontinuous as a map $Y\to\hat X$, we will show that it extends
to a $Y$-$\hat X$-approaching map.

Let $F$ and $\bar F$ denote respectively the maps $f\x\id_{[1,\infty)}\:X\x[1,\infty)\to Y\x[1,\infty)$
and $f\x\id_{[1,\infty]}\:X\x[1,\infty]\to Y\x[1,\infty]$.
Let us fix some metric on $Y$ and let us write $V=\{(y,t)\in Y\x[1,\infty)\mid d(y,B)\le 1/t\}$ and
$\bar V=V\cup B\x\{\infty\}\subset Y\x[1,\infty]$.
Upon identifying $X$ and $Y$ with $X\x\{\infty\}$ and $Y\x\{\infty\}$, let us note that 
$Y$ is a closed homotopy negligible subset of $Y\x [1,\infty]$ and $\hat X$ 
is a closed homotopy negligible subset of $X\x[1,\infty]\cup MC(\bar F_{\bar V})$.
The composition 
\[\Phi\:X\x[1,\infty)\cup MC(F_V)\subset MC(F)\xr{\pi\x\id_{[1,\infty)}} Y\x[1,\infty)\]
is an $\hat X$-$Y$-approaching map continuously extending $\phi$.
On the other hand $\Phi$ has a section $S\:Y\x[1,\infty)\to X\x[1,\infty)\cup MC(F_V)$, defined by 
\[S(y,t)=\begin{cases}
(y,t)\in B\x[1,\infty)\subset MC(F_V)&\text{ if }y\in B,\\
\big(y^*,t,\,d(y,B)/t\big)\in F^{-1}(V)\x (0,1]\subset MC(F_V)&\text{ if }y\notin B\text{ and }d(y,B)\le t,\\
(y^*,t)\in X\x[1,\infty)&\text{ if }d(y,B)\ge t.
\end{cases}\]
Let us note that $s$ provides a continuous extension of $S$.
Given a sequence of points $(y_i,t_i)\in (Y\but B)\x[1,\infty)$ converging to a $y\in B$, 
it is easy to see that all cluster points of the sequence $\big(S(y_i,t_i)\big)$ lie 
in $MC(f_{\{y\}})$.
Writing $\Sigma=\{(y_i,t_i)\mid i\in\N\}$ and $\Sigma^+=\Sigma\cup\{y\}$, 
we get that $S(\Sigma)\cup MC(f_{\{y\}})$ is a closed subset of $MC(\bar F_{\Sigma^+})$.
Since $\Sigma^+$ is compact and $f$ is perfect, $MC(\bar F_{\Sigma^+})$ is compact.
Hence $S(\Sigma)\cup MC(f_{\{y\}})$ is also compact, and therefore the sequence 
$\big(S(y_i,t_i)\big)$ has a cluster point in $MC(f_{\{y\}})\subset\hat X$.
It follows that $S$ is a $Y$-$\hat X$-approaching map.

Next, it is easy to see that $X\x[1,\infty)\cup MC(F_V)$ strong deformation retracts onto
the image of $S$ by a homotopy $h_t$ that keeps $MC(F_{\{v\}})$ within itself for each $v\in V$.
Arguing like before, it is easy to show that $h_t$ is a $\hat X$-$\hat X$-approaching homotopy.
Then by Corollary \ref{negligible2} $Y$ and $\hat X$ are fine shape equivalent, and more
specifically it follows from Theorem \ref{negligible} that $\phi$ is a fine shape equivalence.
\end{proof}

\begin{theorem}\label{fs-me}
Let $H$ be a (co)homology theory on closed pairs of metrizable spaces.
If $H$ is fine shape invariant, then it satisfies the map excision axiom.
\end{theorem}

\begin{proof} Let $f\:(X,A)\to(Y,B)$ be a closed map that restricts to a homeomorphism between $X\but A$ 
and $Y\but B$.
We need to show that $f$ induces isomorphisms on the theory $H$.
By \cite{M00}*{Lemma \ref{book:perfect}} we may assume without loss of generality that $f$ is perfect.
Now $f$ factors as $(X,A)\xr{i}\big(X\x\{0\}\cup A\x I,\,A\x I\big)\xr{j}
\big(X\cup MC(f_B),\,MC(f_B)\big)\xr{\phi}(Y,B)$, where $i$ is the obvious inclusion, 
$j$ is an embedding onto a ``half'' of the metric mapping cylinder, and $\phi$ is 
the restriction of the standard map $\pi\:MC(f)\to Y$.
Now $i$ is a homotopy equivalence of pairs, and so induces isomorphisms on $H$ by 
the homotopy axiom; and $j$ induces isomorphisms on $H$ by the usual excision axiom.
By Lemma \ref{fish-MC} $\phi\:X\cup MC(f_B)\to Y$ is a fine shape equivalence.
Then by the hypothesis it induces isomorphisms on $H$.
On the other hand, the restriction $MC(f_B)\to B$ of $\phi$ is clearly a homotopy equivalence, 
and so induces isomorphisms on $H$ by the homotopy axiom.
Hence by the $5$-lemma $\phi\:\big(X\cup MC(f_B),\,MC(f_B)\big)\to(Y,B)$ also induces isomorphisms 
on $H$.
\end{proof}

\subsection{Fine Z-sets}

In order to establish some results (Theorem \ref{main1a}(b), Corollary \ref{homology}, Theorem \ref{ash-surj}), 
it will be convenient to work with ``fine'' or ``semi-fine'' Z-sets.

We say that a subset $X$ of an absolute retract $M$ is {\it a semi-fine Z-set} in $M$ if $X$ is a Z-set 
in $M$ and each compact subset $K\subset X$ lies in a compact absolute retract $M_K\subset M$ such that 
$M_K\cap X=K$ and $K$ is a Z-set in $M_K$.

\begin{lemma} \label{fineZ0} 
Every metrizable space $X$ embeds as a semi-fine Z-set in some absolute retract $M$.
If $X$ is completely metrizable and/or separable, then so is $M$.
\end{lemma}

\begin{proof} If $X$ is a singleton, it suffices to consider $M=[0,1]$, where $X$ 
is identified with $\{0\}$.
Thus we may assume that $X$ contains at least two points.
Then $X$ is a Z-set in the space $P(X)$ of probability measures, which is an absolute retract;
moreover, if $X$ is completely metrizable and/or separable, then so is $P(X)$; and for every compact subset 
$K\subset X$ the space $P(K)$ is a compact absolute retract which lies in $P(X)$ and meets $X$ in $K$
(see \cite{M00}*{Theorem \ref{book:prob}}).
\end{proof}

We say that a subset $X$ of an absolute retract $M$ is {\it a fine Z-set} in $M$ if $X$ is a Z-set in $M$ and 
each compact subset $K\subset M$ lies in a compact absolute retract $M_K\subset M$ such that $M_K\cap X=K\cap X$ 
and $M_K\cap X$ is a Z-set in $M_K$.
Let us note that for the empty set to be a fine Z-set in a given absolute retract $M$ is a non-vacuous condition on $M$.

\begin{lemma} \label{fineZ}
Every separable metrizable space $X$ embeds as a fine Z-set in some separable 
absolute retract $M$.
Moreover, if $X$ is completely metrizable or locally compact, then so is $M$.
\end{lemma}

\begin{proof} Let $\bar X$ be the completion of $X$ in some metric.
If $X$ is completely metrizable, then we may assume that $\bar X=X$.
If $X$ is a local compactum, then it is completely metrizable, so in this case also $\bar X=X$.

By Isbell's theorem $\bar X$ is homeomorphic to the limit of a scalable inverse sequence 
$\dots\xr{f_2}|K_2|\xr{f_1}|K_1|$ of finite dimensional separable locally compact polyhedra;
if $X$ is locally compact, then each $f_i$ may be assumed to be proper; moreover, 
there exist admissible subdivisions $K_i'$ of the simplicial complexes $K_i$ such that 
each $f_i$ is triangulated by a simplicial map $K_{i+1}\to K_i'$
(see \cite{M00}*{Theorem \ref{book:isbell}}).
By augmenting the inverse sequence we may assume that $|K_1|=pt$; in the case where $X$
locally compact we may assume that $|K_1|=[0,\infty)$ (and the bonding maps remain proper).
Then by \cite{M00}*{Theorem \ref{book:extended telescope}} the extended mapping telescope 
$|K|_{[0,\infty]}$ is a separable absolute retract and $\bar X$ is fine Z-set in it.
In the case where $X$ is locally compact, so is $|K|_{[0,\infty]}$.
Thus we are done in the locally compact and completely metrizable cases.

In the general case let $Y=\bar X\but X$.
Since $|K|_{[0,\infty)}$ is homotopy dense in $|K|_{[0,\infty]}$
(see \cite{M00}*{Proposition \ref{book:telescope retraction}}), so is 
$M\bydef |K|_{[0,\infty]}\but Y$.
Hence $M$ is an absolute retract (see \cite{M00}*{Theorem \ref{book:homotopy dense}}).
Since $\bar X$ is a fine Z-set in $|K|_{[0,\infty]}$, it is easy to see that $X$ 
is a fine Z-set in $M$.
\end{proof}

\begin{remark} \label{fineZ'}
(a) Every Z-set in $\R^\infty$, in the space $c_{00}$ of all eventually zero sequences with 
the $l_\infty$ metric, or in $c_{00}\x I^\infty$ is a fine Z-set in it 
\cite{M00}*{Corollary \ref{book:fineZset}(a)}.

(b) By (a) and \cite{M00}*{Proposition \ref{book:compactZset property}} every compact subset of
$\R^\infty$, of $c_{00}$ or of $c_{00}\x I^\infty$ is a fine Z-set in it.

(c) By (a) and Remark \ref{Z-embeddings}(b) every Polish space $X$ embeds as a fine Z-set in $\R^\infty$.

(d) If $P$ is a metrizable space of the form $P=\bigcup_i R_i$, where each $R_i$ is a compact 
absolute retract and each $R_i\subset\Int R_{i+1}$, then every closed subset of $P\x\{0\}$
is a fine Z-set in $P\x I$ \cite{M00}*{Corollary \ref{book:manifold boundary}}.
Hence we get the following from Remark \ref{Z-embeddings}(c)-(f):
\begin{itemize}
\item Every compactum embeds as a fine Z-set in the Hilbert cube $I^\infty$;
\item Every local compactum embeds as a fine Z-set in $I^\infty\x[0,\infty)$;
\item Every $n$-dimensional compactum embeds as a fine Z-set in $I^{2n+2}$;
\item Every $n$-dimensional local compactum embeds as a fine Z-set in $I^{2n+1}\x[0,\infty)$.
\end{itemize}
\end{remark}

\begin{remark}
Let $X$ be a fine Z-set in an absolute retract $M$.
Then 
\[\tau_M(X)\bydef \{P\but X\mid P\text{ is a compact AR in $M$ such that $P\cap X$ 
is a Z-set in $P$}\}\]
is a cofinal set in the directed family $\kappa_M(X)$.
In particular, $\tau_M(X)$ is itself a directed family of subsets of $M\but X$.

Since $\tau X$ is cofinal in $\kappa X$, a continuous map $f\:M\but X\to N\but Y$, where $X$ is a fine Z-set
in $M$ and $Y$ is a fine Z-set in $N$, is an $X$-$Y$-approaching map if and only if it is a
filtered map $(M\but X)_{\tau X}\to (N\but Y)_{\tau Y}$ and its restriction to every member of $\tau X$ is
a proper map.
Let us note, however, that $f_\star\:\tau X\to\tau Y$ is not monotone in general.
\end{remark}

\section{Comparison} \label{comparison}

\subsection{Antishape}

We refer to \cite{M1} for a thorough treatment of strong shape of compacta.

\begin{proposition} \label{compact-ss}
Let $X$ and $Y$ be compacta.
Then there is a bijection between the set of fine shape morphisms $X\to Y$ and the set of strong shape morphisms $X\to Y$ 
which respects composition of mophisms and the identity morphisms.
\end{proposition}

\begin{proof}
This is more or less of a tautology, depending on how one defines strong shape morphisms 
of compacta; but see \cite{M00}*{Theorem \ref{book:ANR-ss}} for a relation between the definition 
that interests us and what is perhaps the best known definition.
\end{proof}

Let $\Ssh$ be the category whose objects are compacta and whose morphisms are strong shape morphisms.
If $X$ is a metrizable space, let $\K_X$ be the poset of its compact subsets ordered by inclusion
and let $\vec\K_X$ be the associated direct system $(K,\iota^K_L;\K_X)$, where $\iota^K_L\:K\to L$
is the inclusion map between $K,L\in\K_X$ such that $K\subset L$.
If $X$ and $Y$ are metrizable spaces, an {\it antishape morphism} (=compactly generated strong shape 
morphism) $X\to Y$ is an $\ind\Ssh$-morphism $\vec\K_X\to\vec\K_Y$, where ``$\ind$'' stands 
for the ind-category.
(This definition is spelled out it in more elementary terms in \cite{M00}*{\S\ref{book:antishape}}.)
The set of antishape morphisms between metrizable spaces $X$ and $Y$ is denoted by $[X,Y]_\Ash$.
The {\it antishape category} has metrizable spaces as its objects and their antishape morphisms as 
its morphisms.

\begin{theorem} \label{antishape}
(a) There is a functor $\F^\Fish_\Ash$ from the fine shape category to the antishape category 
which sends every object to itself.

(b) The forgetful functor from the homotopy category of metrizable spaces to the antishape category 
is the composition of $\F^\Ho_\Fish$ and $\F^\Fish_\Ash$.
\end{theorem} 

There is hardly anything unexpected in the proof, but the result has important applications
(see Corollary \ref{homology}), so it seems worthwhile to have the details written out.

\begin{proof}
We need to show that for metrizable spaces $X$ and $Y$ the forgetful map $[X,Y]\to [X,Y]_\Ash$ 
factors through $[X,Y]_\Fish$, and moreover it does so in a way that respects composition
of morphisms and the identity morphisms.

Let $M$ and $N$ be absolute retracts containing $X$ and $Y$ as semi-fine Z-sets (see Lemma \ref{fineZ0}).
Thus every compactum $K\subset X$ is contained as a Z-set in a compact absolute retract $M_K\subset M$ 
such that $M_K\cap X=K$, and every compactum $L\subset Y$ is contained as a Z-set in 
a compact absolute retract $N_L\subset N$ such that $N_L\cap Y=L$.
Let us write $M_K^\circ=M_K\but K$ and $N_L^\circ=N_L\but L$.

Let $G\:M\but X\to N\but Y$ be an approaching map.
Let us fix some choice of the map $G_\star\:\kappa_\text{AR} X\to\kappa_\text{AR} Y$ between 
the compact AR filtrations.
Given a $K\in\K_X$, we have $G_\star(M_K^\circ)=N_{L_K}^\circ$ for some $L_K\in\K_Y$.
Hence the restriction $G_K\:M_K^\circ\to N_{L_K}^\circ$ of $G$ determines a strong shape morphism 
$g_K\:K\to L_K$.
Given a $K'\in\K_X$ containing $K$, since $\kappa_\text{AR}Y$ is directed, there exists an $L\in\K_Y$ 
such that $N_{L_K}\cup N_{L_{K'}}\subset N_L$.
Since $G_K$ and $G_{K'}$ are restrictions of the same map $G$, the diagram
\[\begin{tikzcd}[row sep=0.5em,column sep=1.8em]
M_{K'}^\circ\ar[r, "G_{K'}"] & N_{L_{K'}}^\circ \ar[rd,hook] & \\
& &N_L^\circ \\
M_K^\circ\ar[r, "G_K"] \ar[uu,hook] & N_{L_K}^\circ \ar[ru,hook] &         
\end{tikzcd}\]
commutes (even strictly, not just up to proper homotopy).
Hence the diagram
\[\begin{tikzcd}[row sep=0.5em,column sep=1.8em]
K' \ar[r, "g_{K'}"] & L_{K'} \ar[rd,hook] & \\
& & L \\
K\ar[r, "g_K"] \ar[uu,hook] & L_K \ar[ru,hook] &
\end{tikzcd}\]
also commutes.
Thus $\big(g_K\:K\to L_K\big)_{K\in\K_X}$ is a $\dir\Ssh$-morphism $\G\:\vec\K_X\to\vec\K_Y$.

Now let $H\:(M\but X)\x I\to N\but Y$ be an approaching homotopy between
approaching maps $G,\tilde G\:M\but X\to N\but Y$.
Let us fix some choice of the map $H_\star\:\kappa_\text{AR} X\x I\to\kappa_\text{AR} Y$.
A $\dir\Ssh$-morphism $\tilde\G\:\vec\K_X\to\vec\K_Y$ is defined similarly to $\G$ using 
the approaching map $\tilde G$.
Given a $K\in\K_X$, the corresponding components of $\G$ and $\tilde\G$ are 
$g_K\:K\to L_K$ (defined above) and $\tilde g_K\:K\to\tilde L_K$ (defined similarly),
which are in turn given by the restrictions $G_K\:M_K^\circ\to N_{L_K}^\circ$
and $\tilde G_K\:M_K^\circ\to N_{\tilde L_K}^\circ$ of $G$ and $\tilde G$.
By Lemma \ref{approaching homotopy} $H_\star(M_K^\circ)$ lies in $N_L^\circ\x I$ 
for some $L\in\K_Y$.
Then in particular $N_{L_K}\cup N_{\tilde L_K}\subset N_L$.
Moreover, the diagram
\[\begin{tikzcd}[row sep=0.6em,column sep=1.2em]
& N_{L_K}^\circ \ar[rd,hook] & \\
M_K^\circ \ar[ru, "G_K" near end]\ar[rd, "\tilde G_K"' near end]  & & N_L^\circ \\
& N_{\tilde L_K}^\circ \ar[ru,hook] &         
\end{tikzcd}\]
commutes up to proper homotopy, using the restriction $H_K\:M_K^\circ\x I\to N_L^\circ$ of $H$.
Hence the diagram
\[\begin{tikzcd}[row sep=0.6em,column sep=1.2em]
& L_K \ar[rd,hook] & \\
K \ar[ru, "g_K" near end]\ar[rd, "\tilde g_K"' near end]  & & L \\
& \tilde L_K \ar[ru,hook] &         
\end{tikzcd}\]
also commutes.
Thus $\G$ and $\tilde\G$ represent the same $\ind\Ssh$-morphism $\vec\K_X\to\vec\K_Y$.

Now let $f\:X\to Y$ be a continuous map. 
By Lemma \ref{Milnor}(a) it extends to a continuous map $F\:M\to N$ such that $F^{-1}(Y)=X$.
Let $\F=\big(f_K\:K\to L_K\big)_{K\in\K_X}$ be given by the preceding construction applied to
the restrictions $F_K\:M_K^\circ\to N_{L_K}^\circ$ of $F$, where each $L_K\in\K_Y$ is chosen
so that $F(M_K)\subset N_{L_K}$.
Each $f_K$ is the image of the homotopy class of $f|_K\:K\to L_K$ under the forgetful map
$[K,L_K]\to[K,L_K]_\Ssh$.
On the other hand, let $f'_K\:K\to f(K)$ be the image of the homotopy class of 
$f|_K\:K\to f(K)$ under the forgetful map $[K,f(K)]\to[K,f(K)]_\Ssh$.
Then the $\dir\Ssh$-morphism $\F'\bydef (f'_K)_{K\in\K_X}$ represents the image of $[f]$
under the map $[X,Y]\to[X,Y]_\Ash$.
But clearly $\F$ and $\F'$ represent the same $\ind\Ssh$-morphism $\vec\K_X\to\vec\K_Y$.

In particular, the preceding discussion applies to the case where $f=\id_X$ 
(but $M\ne N$).
From this it follows that if given approaching homotopy classes $[G]\:M\but X\to N\but X$ and 
$[G']\:M'\but X\to N'\but Y$ represent the same fine shape morphism, then
the corresponding $\ind\Ssh$-morphisms $[\G],[\G']\:\vec\K_X\to\vec\K_Y$ are equal.
Thus we have constructed a map $\F_{XY}\:[X,Y]_\Fish\to [X,Y]_\Ash$.
By the general case of the preceding discussion it sends the fine shape morphism represented 
by a map $f\:X\to Y$ to the antishape morphism represented by $f$.
Thus the forgetful map $[X,Y]\to [X,Y]_\Ash$ factors through $\F_{XY}$.
Also we get that $\F_{XX}([\id_X])=[\id_X]$.
It is straightforward to verify that for fine shape morphisms $X\xr{f}Y\xr{g}Z$ we have
$\F_{YZ}(g)\F_{XY}(f)=\F_{XZ}(gf)$.
\end{proof}

Theorem \ref{antishape}(a) immediately yields

\begin{corollary} Fine shape equivalence implies antishape equivalence.
\end{corollary}

\begin{example} The forgetful map $[X,S^n]_\Fish\to [X,S^n]_\Ash$ need not be injective
for $n\ge 1$ and locally compact separable metrizable $X$.

Indeed, by Theorem \ref{ANR} the forgetful map $[X,Q]\to[X,Q]_\Fish$ is bijective
as long as $Q$ is an ANR.
However, the forgetful map $[X,S^n]\to [X,S^n]_\Ash$ is not injective for $n\ge 2$,
when $X$ is the mapping telescope of a direct sequence $S^{n-1}\xr{p} S^{n-1}\xr{p}\dots$ 
of degree $p$ maps, and also for $n=1$, when $X$ is the subset
$\big(\{0\}\cup\{\pm\frac1n\mid n\in\N\}\big)\x[-1,1]\but\{(0,0)\}$ of the plane 
\cite{M00}*{Example \ref{book:ash-anr}}.
\end{example}

\subsection{Anti-resolutions} 

Let $X$ be a local compactum (=locally compact separable metrizable space).
A (sequential) {\it anti-resolution} of $X$ is a chain 
$\M$ of subsets $M_0\subset M_1\subset\dots$ of a metrizable space $M$ 
containing $X$ such that
\begin{enumerate}
\item each $M_i$ is a compact absolute retract;
\item each $K_i\bydef M_i\cap X$ is a Z-set in $M_i$;
\item $\K^\M_X\bydef \{K_0,K_1,\dots\}$ is a cofinal subset of $\K_X$,
\end{enumerate}
where $\K_X$ is the poset of all compact subsets of $X$, ordered by inclusion.
We could more accurately define an anti-resolution to be a pair $(M,\M)$ satisfying these
conditions, but in fact there will be no need to mention $M$, so it is unnecessary 
to reserve a letter for it.

On the other hand, we call $\M$ a {\it $\Fish$-resolution} of $X$ if the following hold:
\begin{enumerate}
\item[(1$'$)] $M$ is an absolute retract;
\item[(2$'$)] $X$ is a Z-set in $M$;
\item[(3$'$)] $\kappa_\M X\bydef \{M_0\but K_0,\,M_1\but K_1,\dots\}$ is a cofinal subset of $\kappa_M X$.
\end{enumerate}
We could more accurately define a $\Fish$-resolution to be a pair $(M,\M)$ satisfying these
conditions, but in fact $M$ is determined by $\M$ since (3$'$) implies that $M=\bigcup_i M_i$.

\begin{lemma} \label{anti-resolution}
Every local compactum $X$ admits an anti-resolution which is also its $\Fish$-resolution.
\end{lemma}

\begin{proof}
$X$ is homeomorphic to a closed subset of $I^\infty\x[0,\infty)$, where $I^\infty$ 
is the Hilbert cube (see \cite{M00}*{Proposition \ref{book:lc-emb}}).
So we may identify $X$ with a closed subset of $I^\infty\x[0,\infty)\x\{0\}$,
which is a Z-set in the absolute retract $M\bydef I^\infty\x[0,\infty)\x I$.
Each $M_i\bydef I^\infty\x[0,i]\x I$ is a compact absolute retract and 
$M_0\subset M_1\subset\dots$.
Each $K_i\bydef X\cap M_i$ lies in $I^\infty\x[0,\infty)\x\{0\}$ and hence is a Z-set in $M_i$.
Since $M=\bigcup_i M_i$ and each $M_i\subset\Int M_{i+1}$, the set $\{M_0,M_1,\dots\}$ 
is cofinal in $\K_M$ (see \cite{M00}*{Proposition \ref{book:local compactum}}).
Hence $\{K_0,K_1,\dots\}$ is cofinal in $\K_X$ and $\{M_0\but K_0,\,M_1\but K_1,\dots\}$ 
is cofinal in $\kappa_M X$.
\end{proof}

\begin{proof}[Alternative proof]
Let $M$ be any locally compact separable absolute retract containing $X$ as a fine Z-set
(see Lemma \ref{fineZ}).
Thus every compactum $K\subset M$ lies in a compact absolute retract $M_K\subset M$ such that 
$M_K\cap X=K\cap X$ and $M_K\cap X$ is a Z-set in $M_K$.
Since $X$ is a local compactum, there exists a chain $K_0\subset K_1\subset\dots$ of compact subsets 
of $X$ such that $X=\bigcup_i K_i$ and each $K_i\subset\Int K_{i+1}$; moreover,
this chain is cofinal in $\K_X$ (see \cite{M00}*{Proposition \ref{book:local compactum}}).
On the other hand, since $M\but X$ is a local compactum, it the union of an increasing chain 
$Q_0\subset Q_1\subset\dots$ of compacta.

Let $M_0=M_{K_0\cup Q_0}$.
Assuming that $M_i$ is defined, let $N_i$ be a compact neighborhood of $M_i$, disjoint from
$X\but\Int K_{i+1}$, and let $M_{i+1}=M_{N_i\cup K_{i+1}\cup Q_{i+1}}$.
Then each $M_i\cap X=K_i$ and $K_i$ is a Z-set in $M_i$.
Moreover, each $M_i\subset\Int M_{i+1}$.
Finally, $\bigcup_i M_i$ contains both $\bigcup_i K_i=X$ and $\bigcup_i Q_i=M\but X$,
and so equals $M$.
Hence $\{M_0,M_1,\dots\}$ is cofinal in $\K_M$.
Consequently $\{M_0\but K_0,\,M_1\but K_1,\dots\}$ is cofinal in $\kappa_M X$.
\end{proof}

\begin{lemma} \label{telescope-zset}
Let $\M=\big(M_0\subset M_1\subset\dots\big)$ be an anti-resolution of 
a local compactum $X$, and let $K_i=M_i\cap X$.
Let $K_{[0,\infty)}$ and $M_{[0,\infty)}$ be the mapping telescopes of the chains 
$K_0\subset K_1\subset\dots$ and $M_0\subset M_1\subset\dots$.

Then $\M'\bydef \big(M_{[0,1]}\subset M_{[0,2]}\subset\dots\big)$ is a $\Fish$-resolution
of $K_{[0,\infty)}$.
\end{lemma}

\begin{proof}
Every finite telescope $M_{[0,n]}$ deformation retracts onto $M_n$, which is contractible.
Hence $M_{[0,\infty)}$ has trivial homotopy groups.
On the other hand, $M_{[0,\infty)}$ is an ANR (see \cite{M00}*{Corollary \ref{book:telescope-ANR}}).
In particular, it is homotopy equivalent to a polyhedron.
Hence $M_{[0,\infty)}$ is contractible by Whitehead's theorem.
Therefore $M_{[0,\infty)}$ is an absolute retract.

Since each $K_i$ is a Z-set in $M_i$, there exists a homotopy $\phi_t^i\:M_i\to M_i$ 
such that $\phi_0^i=\id_{M_i}$ and $\phi_t^i(M_i)\subset M_i\but K_i$ for all $t>0$.
Let us define $\Phi_t^i\:M_i\x I\to M_{i+1}$ by
$\Phi_t^i(x,s)=\phi_{\max(t-s,0)}^{i+1}\big(\phi_{\min(t,s)}^i(x)\big)$.
Then for each $x\in M_i$ we have $\Phi_t^i(x,0)=\phi_t^{i+1}(x)$ and $\Phi_t^i(x,1)=\phi_t^i(x)$, and 
for $t\le s$ we have $\Phi_t^i(x,s)=\phi_t^i(x)\in M_i$.
It follows that we may define $\Phi_t\:M_{[0,\infty)}\to M_{[0,\infty)}$, where 
$M_{[0,\infty)}=\bigcup_i M_i\x[i,i+1]\subset M\x [0,\infty)$, by 
$\Phi_t(x,i+1-s)=\big(\Phi_t^i(x,s),\,i+1-s+t\big)$ for $s\in [0,1]$.
Now for each $x\in M_i$ we also have $\Phi_0^i(x,s)=x$ and $\Phi_t^i(x,s)\in M_{i+1}\but K_{i+1}$ 
for all $t>0$ (indeed, if $s\ge t$, then $\min(t,s)=t>0$; and if $s<t$, then $\max(t-s,0)=t-s>0$).
Consequently for each $y\in M_{[0,\infty)}$ we have $\Phi_0(y)=y$ and 
$\Phi_t(y)\in M_{[0,\infty)}\but K_{[0,\infty)}$ for all $t>0$.
Thus $K_{[0,\infty)}$ is a Z-set in $M_{[0,\infty)}$.

Since each $M_i$ is compact, so is $M_{[0,i]}$.
Clearly each $M_{[0,i]}\subset\Int M_{[0,i+1]}$.
Hence $\M'\bydef \big(M_{[0,1]}\subset M_{[0,2]}\subset\dots\big)$ is cofinal in $\K_{M_{[0,\infty)}}$.
Consequently $\kappa_{\M'}K_{[0,\infty)}$ is cofinal in $\kappa_{M_{[0,\infty)}}K_{[0,\infty)}$.
\end{proof}

\begin{theorem} \label{ash-surj}
If $X$ is a local compactum (=locally compact separable metrizable space), then the forgetful map 
$[X,Y]_\Fish\to [X,Y]_\Ash$ is surjective for every separable metrizable space $Y$.
\end{theorem}

Theorems \ref{ANR} and \ref{ash-surj} imply

\begin{corollary} \label{lc-anr} \cite{M00}*{Theorem \ref{book:lc-anr}}
If $X$ is a local compactum and $Q$ is a separable ANR, then the forgetful map $[X,Q]\to[X,Q]_\Ash$ is surjective.
\end{corollary}

The proof of Theorem \ref{ash-surj} is modeled after the original proof of 
Corollary \ref{lc-anr}.

\begin{proof}[Proof of Theorem \ref{ash-surj}]
By Lemma \ref{anti-resolution} $X$ has an anti-resolution $\M=\big(M_0\subset M_1\subset\dots\big)$
which is also its $\Fish$-resolution.
Let $K_i=M_i\cap X$ and let us write $M_i^\circ=M_i\but K_i$.
In addition, let $N$ be an absolute retract containing $Y$ as a fine Z-set (see Lemma \ref{fineZ};
this is the only place in the proof where we need $Y$ to be separable).
Thus every compactum $L\subset N$ lies in a compact absolute retract $N_L\subset N$ such that 
$N_L\cap Y=L\cap Y$ and $N_L\cap Y$ is a Z-set in $N_L$.
Let us write $N_L^\circ=N_L\but L$.

Given an antishape morphism $f\:X\to Y$, it is represented by a $\dir\Ssh$-morphism 
$\F\:\vec\K^\M_X\to\vec\K_Y$ whose components $f_i\:K_i\to L_{K_i}$ are strong shape
morphisms that are represented by some proper maps $F_i\:M_i^\circ\to N_{L_{K_i}}^\circ$.
Since the $f_i$ form a $\dir\Ssh$-morphism, for each $i$ there exists an $L_i\in\K_Y$
containing $L_{K_i}\cup L_{K_{i+1}}$ such the diagram
\[\begin{tikzcd}[row sep=0.5em,column sep=1.8em]
M_{i+1}^\circ \ar[r, "F_{i+1}"] & N_{L_{K_{i+1}}}^\circ 
\ar[rd,hook] & \\
& & N_{L_i}^\circ \\
M_i^\circ\ar[r, "F_i"] \ar[uu,hook] & N_{L_{K_i}}^\circ 
\ar[ru,hook] &
\end{tikzcd}\]
commutes up to proper homotopy.
Let $H_i\:M_i^\circ\x I\to N_{L_i}^\circ$ be this proper homotopy.
Also let $N_0=N_{L_0}$ and for each $i\ge 0$ let $N_{i+1}=N_{N_i\cup N_{L_{i+1}}}$.
Then each $N_i$ is a compact absolute retract, $N_i\cap Y=L_i$ and $L_i$ is a Z-set in $N_i$.
Moreover, $N_0\subset N_1\subset\dots$ and each $N_{L_i}\subset N_i$.
Let us write $N_i^\circ=N_i\but Y$.

Let $K_{[0,\infty)}$ and $M_{[0,\infty)}$ be the mapping telescopes of the chains 
$K_0\subset K_1\subset\dots$ and $M_0\subset M_1\subset\dots$.
Then $M_{[0,\infty)}^\circ\bydef M_{[0,\infty)}\but K_{[0,\infty)}$ is the mapping telescope of the chain
$M_0^\circ\subset M_1^\circ\subset\dots$.
By Lemma \ref{telescope-zset} $M_{[0,\infty)}$ is an absolute retract containing $K_{[0,\infty)}$ as
a Z-set, and the chain $M_{[0,1]}^\circ\subset M_{[0,2]}^\circ\subset\dots$ is cofinal in
$\kappa_{M_{[0,\infty)}}K_{[0,\infty)}$.

The maps $F_i$ and the homotopies $H_i$ combine into a map
$F\:M_{[0,\infty)}^\circ\to N\but Y$.
Each $F|_{M_{[0,n]}^\circ}$ is a proper map of $M_{[0,n]}^\circ$ into $N_n^\circ$.
Hence $F$ is a $K_{[0,\infty)}$-$Y$-approaching map.

Since $X$ is a local compactum, there exists a homotopy equivalence 
$\phi\:X\to K_{[0,\infty)}$ which is homotopy inverse to the projection $\pi\:K_{[0,\infty)}\to X$
(see \cite{M00}*{Proposition \ref{book:lc-telescope-a}}).
By Theorem \ref{fsm}(a) $\phi$ combines with some $X$-$K_{[0,\infty)}$-approaching map 
$\Phi\:M\but X\to M_{[0,\infty)}^\circ$ into a continuous map $M\to M_{[0,\infty)}$. 
In particular, $\Phi$ sends each $M_i^\circ$ into $M_{[0,n_i]}^\circ$ for some $n_i\in\N$.
The composition $G\:M\but X\xr{\Phi} M_{[0,\infty)}^\circ\xr{F} N\but Y$ is an $X$-$Y$-approaching map.
It sends each $M_i^\circ$ into $N_{n_i}^\circ$.

The forgetful map%
\footnote{The forgetful map $[X,Y]_\Fish\to [X,Y]_\Ash$ was constructed in the proof of
Theorem \ref{antishape}, where $X$ was assumed to be a semi-fine Z-set in $M$.
In the current proof we are not assuming $X$ to be a semi-fine Z-set in $M$, although we could
have assumed this (using the alternative proof of Lemma \ref{anti-resolution}).
But in fact we do not really need this assumption, because we are assuming that each
$X_i$ is a Z-set in $M_i$ and that the chain $M_0^\circ\subset M_1^\circ\subset\dots$
is cofinal in $\kappa_M X$, and this is enough for the proof of Theorem \ref{antishape}
to go through.
Strictly speaking, this modified proof of Theorem \ref{antishape} in the case where $X$
is locally compact yields a new definition of the forgetful map $[X,Y]_\Fish\to [X,Y]_\Ash$;
but it also works to show (taking into account the alternative proof of 
Lemma \ref{anti-resolution}) that this new map equals the original one.}
$[X,Y]_\Fish\to [X,Y]_\Ash$ sends $[G]$ to the $\ind\Ssh$-morphism
represented by the $\dir\Ssh$-morphism $\G\:\vec\K^\M_X\to\vec\K_Y$ whose components 
$g_i\:K_i\to L_{n_i}$ are strong shape morphisms that are represented by the proper maps
$G|_{M_i^\circ}\:M_i^\circ\to N_{n_i}^\circ$.

Let $S_i$ denote the inclusion $M_i^\circ=M_i^\circ\x\{i\}\subset M_{[0,\infty)}^\circ$
and let $\Pi$ denote the projection $M_{[0,\infty)}^\circ\to M\but X$.
The composition $\Pi S_i\:M_i^\circ\to M\but X$ is the inclusion map, so 
$G|_{M_i^\circ}=G\Pi S_i=F\Phi\Pi S_i$.
On the other hand, since $\phi$ is a homotopy equivalence, the composition 
$K_{[0,\infty)}\xr{\pi}X\xr{\phi}K_{[0,\infty)}$ is homotopic to $\id_{K_{[0,\infty)}}$;
consequently by Theorem \ref{fsm}(b) $\Phi\Pi\:M_{[0,\infty)}^\circ\to M_{[0,\infty)}^\circ$ 
is homotopic to $\id_{M_{[0,\infty)}^\circ}$ by a $K_{[0,\infty)}$-$K_{[0,\infty)}$-approaching 
homotopy $H_t\:M_{[0,\infty)}^\circ\to M_{[0,\infty)}^\circ$.
Thus for each $i$ there exists an $m_i$ such that $H_t(M_{[0,i]}^\circ)\subset M_{[0,m_i]}^\circ$ 
for all $t\in I$.
Hence $G|_{M_i^\circ}=F\Phi\Pi S_i$ is homotopic to $FS_i=F_i$ by the proper homotopy 
$FH_tS_i\:M_i^\circ\to N_{m_i}^\circ$.
In other words, the diagram
\[\begin{tikzcd}[row sep=1.2em,column sep=1.2em]
& N_{L_{K_i}}^\circ \ar[rd,hook]
\ar[r,hook] & N_i^\circ \ar[d,hook]\\
M_i^\circ \ar[ru, "F_i" near end]\ar[rd, "G|_{M_i^\circ}"' near end]  & & N_{m_i}^\circ \\
& N_{n_i}^\circ \ar[ru,hook] &         
\end{tikzcd}\]
commutes up to proper homotopy (here $m_i\ge n_i$ due to $H_0=\Phi\Pi$ and $m_i\ge i$ due to $H_1=\id$).
Hence the diagram
\[\begin{tikzcd}[row sep=1.2em,column sep=1.2em]
& L_{K_i} \ar[rd,hook] 
\ar[r,hook] & L_i \ar[d,hook] \\
K_i \ar[ru, "f_i" near end]\ar[rd, "g_i"' near end]  & & L_{m_i} \\
& L_{n_i} \ar[ru,hook] &         
\end{tikzcd}\]
also commutes.
Thus $\F$ and $\G$ represent the same $\ind\Ssh$-morphism $\vec\K^\M_X\to\vec\K_Y$.
\end{proof}

\subsection{Sequential strong antishape} \label{sash}

Let us define strong antishape morphisms between local compacta $X$ and $Y$.
Let $\M=(M_0\subset M_1\subset\dots)$ and $\NN=(N_0\subset N_1\subset\dots)$ be anti-resolutions 
of $X$ and of $Y$, respectively.
Let $K_i=M_i\cap X$ and $L_i=N_i\cap Y$.
Let us write $M_i^\circ=M_i\but K_i$ and $N_i^\circ=N_i\but L_i$.
Let $\kappa_\M X=\{M_0^\circ,M_1^\circ,\dots\}$ and $\kappa_\NN Y=\{N_0^\circ,N_1^\circ,\dots\}$.

A {\it coherent map} $\F\:\kappa_\M X\to\kappa_\NN Y$ consists of (i) a sequence $k_0<k_1<\dots$ of 
natural numbers, (ii) a sequence of proper maps $f_i\:M_i^\circ\to N_{k_i}^\circ$ 
such that the diagrams
\[\begin{tikzcd}[row sep=2em]
M_{i+1}^\circ\ar[r,"f_{i+1}"]&N_{k_{i+1}}^\circ\\
M_i^\circ\ar[r,"f_i"']\ar[u,hook]&N_{k_i}^\circ.\ar[u,hook]
\end{tikzcd}\]
commute up to proper homotopy, via some proper homotopies $F_i\:M_i^\circ\x I\to N_{k_{i+1}}^\circ$;
and (iii) these proper homotopies $F_i$.
We will denote such an $F_i$ graphically by
\[\begin{tikzcd}[row sep=2em]
M_{i+1}^\circ\ar[r,"f_{i+1}"]&N_{k_{i+1}}^\circ\\
M_i^\circ\ar[r,"f_i"']\ar[u,hook]&N_{k_i}^\circ.\ar[u,hook]
\ar[ul,phantom,"{}_{F_i}\hspace{-2pt}\rotatebox{135}{$\Leftarrow$}"]
\end{tikzcd}\]
The coherent map $\F\:\kappa_\M X\to\kappa_\NN Y$ is called a {\it coextension} of a continuous map 
$f\:X\to Y$ if the following holds for each $i\in\N$: (i) $f_i$ and $f|_{K_i}$ are the restrictions of 
a continuous map $M_i\to N_{k_i}$;%
\footnote{Let us note that for this to hold $f(K_i)$ must lie in $L_{k_i}$.} 
(ii) $F_i$ and the composition $K_i\x I\xr{\text{projection}}K_i\xr{f_i}L_{k_i}$ are 
the restrictions of a continuous homotopy $M_i\x I\to N_{k_{i+1}}$. 

On the other hand, we say that the coherent map $\F$ is {\it coherently homotopic} to another coherent map 
$\G\:\kappa_\M X\to\kappa_\NN Y$ consisting of a sequence $(l_i)$, proper maps 
$g_i\:M_i^\circ\to N_{l_i}^\circ$ and proper homotopies
\[\begin{tikzcd}[row sep=2em]
M_{i+1}^\circ\ar[r,"g_{i+1}"]&N_{l_{i+1}}^\circ\\
M_i^\circ\ar[r,"g_i"']\ar[u,hook]&N_{l_i}^\circ.\ar[u,hook]
\ar[ul,phantom,"{}_{G_i}\hspace{-2pt}\rotatebox{135}{$\Leftarrow$}"]
\end{tikzcd}\]
if there exist (i) natural numbers $h_0<h_1<\dots$, where each $h_i\ge\max(k_i,l_i)$, and proper homotopies
\[\begin{tikzcd}[row sep=2em]
M_i^\circ\ar[r,"f_i"]\ar[d,"g_i"']&N_{k_i}^\circ\ar[d,hook]\\
N_{l_i}^\circ\ar[r,hook]&N_{h_i}^\circ,
\ar[ul,phantom,"\rotatebox{45}{$\Leftarrow$}\hspace{-4pt}_{H_i}"]
\end{tikzcd}\]
and (ii) for each $i$ a proper homotopy keeping $M_i^\circ\x\partial I$ fixed between the following 
two proper homotopies $M_i^\circ\x I\to N_{h_{i+1}}^\circ$:
\[\begin{tikzcd}[row sep=1.5em,column sep=1.5em]
M_i^\circ \ar[dddr,phantom,"\rotatebox{47}{$\Leftarrow$}\hspace{-4pt}_{G_i}"]
\ar[drrr,phantom,"\rotatebox{33}{$\Leftarrow$}\hspace{-4pt}_{\bar F_i}"]
\ar[rd,hook] \ar[dd,"g_i"] \ar[rr,"f_i"] & & N_{k_i}^\circ \ar[rd,hook] & \\
& M_{i+1}^\circ \ar[dd,"g_{i+1}"'] \ar[rr,"f_{i+1}"] & & N_{k_{i+1}}^\circ \ar[dd,hook] \\
N_{l_i}^\circ \ar[rd,hook] & & \rotatebox{40}{$\Leftarrow$}\hspace{-4pt}_{H_{i+1}} & \\
& N_{l_{i+1}}^\circ \ar[rr,hook] & & N_{h_{i+1}}^\circ
\end{tikzcd}
\hspace{20pt} \Rrightarrow\hspace{15pt}
\begin{tikzcd}[row sep=1.5em,column sep=1.5em]
M_i^\circ \ar[dd, "g_i"'] \ar[rr, "f_i"] & & N_{k_i}^\circ \ar[dd, hook] \ar[rd, hook] & \\
& \rotatebox{40}{$\Leftarrow$}\hspace{-4pt}_{H_i} & & N_{k_{i+1}}^\circ \ar[dd, hook] \\
N_{l_i}^\circ \ar[rr, hook] \ar[rd, hook] & & N_{h_i}^\circ \ar[rd, hook] & \\
& N_{l_{i+1}}^\circ \ar[rr, hook] & & N_{h_{i+1}}^\circ,\!
\end{tikzcd}\]
where $\bar F_i$ denotes $F_i$ with reversed time, and the graphic notation for stacked 
combination of homotopies is the obvious one (see \cite{M00}*{\S\ref{book:seq-ss}} for the details).

\begin{lemma} \label{coext}
(a) Every map $f\:X\to Y$ admits a coextension $\F\:\kappa_\M X\to\kappa_\NN Y$.

(b) If $\F,\G\:\kappa_\M X\to\kappa_\NN Y$ are coextensions of homotopic maps $f,g\:X\to Y$,
then $\F$ and $\G$ are coherently homotopic.
\end{lemma}

\begin{proof} This follows easily from the compact case of Theorem \ref{fsm}.
\end{proof}

Now let $\M'=(M_0'\subset M_1'\subset\dots)$ and $\NN'=(N_0'\subset N_1'\subset\dots)$ be some other
anti-resolutions of $X$ and of $Y$, respectively.
We say that coherent maps $\F\:\kappa_\M X\to\kappa_\NN Y$ and 
$\F'\:\kappa_{\M'}X\to\kappa_{\NN'}Y$ represent the same {\it strong antishape morphism} 
$X\to Y$ if the diagram
\[\begin{CD}
\kappa_\M X@>\F>>\kappa_\NN Y\\
@V\G_XVV@VV\G_YV\\
\kappa_{\M'}X@>\F'>>\kappa_{\NN'}Y
\end{CD}\]
commutes up to coherent homotopy, where $\G_X$ and $\G_Y$ are some coherent coextensions
of $\id_X$ and $\id_Y$ respectively (their existence is guaranteed by Lemma \ref{coext}(a)). 
To be more precise, the relation that we have just defined on coherent maps is 
a well-defined equivalence relation (by Lemma \ref{coext}(b) it does not depend on the choice 
of the maps $\G_X$ and $\G_Y$ and is symmetric), and strong antishape morphisms $X\to Y$ are defined 
as its equivalence classes.

\begin{theorem} \label{fish-sash} Let $X$ and $Y$ be local compacta.

(a) There is a bijection between $[X,Y]_\Fish$ and the set of strong antishape morphisms $X\to Y$
which respects composition of morphisms and the identity morphisms.

(b) Let us fix some anti-resolutions $\M$ and $\NN$ of $X$ and $Y$ respectively which are also their
$\Fish$-resolutions.
Then there is a bijection between $[X,Y]_\Fish$ and the set of coherent homotopy classes of 
coherent maps $\kappa_\M X\to\kappa_\NN Y$.
\end{theorem}

\begin{proof}[Proof. (b)] Let $M=\bigcup_i M_i$ and $N=\bigcup_i N_i$ where 
$\M=(M_0\subset M_1\subset\dots)$ and $\NN=(N_0\subset N_1\subset\dots)$.
Since $M$ and $N$ are absolute retracts containing $X$ and $Y$ as Z-sets,
by Proposition \ref{representation} there is a bijection between $[X,Y]_\Fish$ 
and the set $[X,Y]_{MN}$ of $X$-$Y$-approaching homotopy classes of $X$-$Y$-approaching maps 
$M\but X\to N\but Y$.

Let $K_{[0,\infty)}$ and $M_{[0,\infty)}$ be as in the statement Lemma \ref{telescope-zset}.
Since $X$ is a local compactum, it is homotopy equivalent to $K_{[0,\infty)}$
(see \cite{M00}*{Proposition \ref{book:lc-telescope-a}}).
Hence it follows from Lemma \ref{telescope-zset} and Theorem \ref{fsm} that there is 
a bijection between $[X,Y]_{MN}$ and the set $[K_{[0,\infty)},Y]_{M_{[0,\infty)}N}$
of $K_{[0,\infty)}$-$Y$-approaching homotopy classes of $K_{[0,\infty)}$-$Y$-approaching maps 
$M_{[0,\infty)}\but K_{[0,\infty)}\to N\but Y$.

Finally, it is not hard to see that $[K_{[0,\infty)},Y]_{M_{[0,\infty)}N}$ can be identified 
with the set of coherent homotopy classes of coherent maps $\kappa_\M X\to\kappa_\NN Y$.
\end{proof}

\begin{proof}[(a)] Let us define an {\it f-strong antishape morphism} $X\to Y$ similarly to a strong
antishape morphism $X\to Y$, but using only those anti-resolutions of $X$ and $Y$ which are also 
their $\Fish$-resolutions.
Then it follows from Lemma \ref{anti-resolution} that there is a bijection between the set
of strong antishape morphisms $X\to Y$ and the set of f-strong antishape morphisms $X\to Y$.

Let $\M$ and $\NN$ be as in (b) and let $M$ and $N$ be as in the proof of (b).
By the proof of (b) there is a bijection between $[X,Y]_{MN}$ and the set of coherent homotopy 
classes of coherent maps $\kappa_\M X\to\kappa_\NN Y$.
Then it follows from Lemma \ref{Milnor}, Theorem \ref{fsm} and Lemma \ref{coext} that there is 
a bijection between $[X,Y]_\Fish$ and the set of f-strong antishape morphisms $X\to Y$.
\end{proof}

\subsection{Shape}

We refer to \cite{M00}*{\S\ref{book:shape}} for a treatment of the shape category, which
is the one that we will follow in the proof of the following theorem.
The set of shape morphisms between metrizable spaces $X$ and $Y$ is denoted by $[X,Y]_\Sh$.

\begin{theorem}\label{shape}
(a) There is a functor $\F^\Fish_\Sh$ from the fine shape category to the shape category 
which sends every object to itself.

(b) The forgetful functor from the homotopy category of metrizable spaces to the shape category 
is the composition of $\F^\Ho_\Fish$ and $\F^\Fish_\Sh$.
\end{theorem}

The proof is not difficult, but the result has important applications
(see Corollary \ref{cohomology}), so it seems worthwhile to have the details written out.

\begin{proof}
We need to show that for metrizable spaces $X$ and $Y$ the forgetful map $[X,Y]\to [X,Y]_\Sh$ 
factors through $[X,Y]_\Fish$, and moreover it does so in a way that respects composition
of morphisms and the identity morphisms.

Let $M$ and $N$ be absolute retracts containing $X$ and $Y$ as Z-sets.
Let $\U$ and $\V$ be the posets of all open neighborhoods of $X$ in $M$ and of $Y$ in $N$,
ordered by inclusion.
Being open subsets of ANRs, their members are themselves ANRs 
(see \cite{M00}*{Lemma \ref{book:ANR-basic}(c)}).
Since $\U$ and $\V$ are codirected, the dual posets $\U^*$ and $\V^*$ are directed.
Let $\cev\U$ and $\cev\V$ be the inverse systems indexed by $\U^*$ and $\V^*$ and formed 
by their members, with inclusions as bonding maps. 
Then $\cev\U$ and $\cev\V$ are ANR resolutions of $X$ and of $Y$
(see \cite{M00}*{Lemma \ref{book:semi-resolution}(a)}).

Since $X$ is a Z-set in $M$, it is a Z-set in each $U\in\U$; moreover,
for each $U\in\U$ there exists a homotopy $\phi_t^U\:M\to M$ such that $\phi_0^U=\id_M$, 
$\phi_t^U|_{M\but U}=\id_{M\but U}$ for all $t\in I$, and $\phi_t^U(M)\subset M\but X$ 
for all $t>0$ \cite{M00}*{Lemma \ref{book:z-set-rel}}.
Given a $W\in\U$, we have the homotopy $\Phi_t^{UW}\:M\to M\but X$ between $\phi_1^W$ and
$\phi_1^U$ defined by $\Phi_t^{UW}(w)=\phi_t^U\big(\phi_{1-t}^W(w)\big)$.
Let us note that $\Phi_t^{UW}|_{M\but (U\cup W)}=\id_{M\but (U\cup W)}$ for all $t\in I$.

Let $G\:M\but X\to N\but Y$ be an approaching map.
Given a $V\in\V$, let $U=G^{-1}(V\but X)\cup X$ and let us consider the composition 
$g_V\:U\xr{\phi_1^U|}U\but X\xr{G|}V\but X\subset V$.
Given an $O\in\V$ that lies in $V$, let $W=G^{-1}(O\but X)\cup X$ and let us note 
that the diagram
\[\begin{tikzcd}
W\ar[r,"g_O"]\ar[d,hook]&O\ar[d,hook]\\
U\ar[r,"g_V"]&V
\end{tikzcd}\]
commutes up to homotopy, using the homotopy
$W\xr{\Phi_t^{UW}|}U\but X\xr{G|}V\but X\subset V$.
Hence $\G\bydef (g_V)_{V\in\V}$ is an $\inv\Ho$-morphism between $\cev\U$ and $\cev\V$.

Now suppose that $H\:(M\but X)\x I\to N\but Y$ is an approaching homotopy between
approaching maps $G,\tilde G\:M\but X\to N\but Y$.
An $\inv\Ho$-morphism $\tilde\G\:\cev\U\to\cev\V$ is defined similarly to $\G$ using 
the approaching map $\tilde G$.
Given a $V\in\V$, the corresponding components of $\G$ and $\tilde\G$ are 
$g_V\:U\to V$ (defined above) and $\tilde g_V\:\tilde U\to V$ (defined similarly).
By Lemma \ref{approaching homotopy} $H^{-1}(V\but Y)$ contains $(W\but X)\x I$ 
for some $W\in\U$.
Then in particular $W\subset U\cap\tilde U$.
Hence the diagram
\[\begin{tikzcd}
W\ar[r,hook]\ar[d,hook]&U\ar[d,"g_V"]\\
\tilde U\ar[r,"\tilde g_V"]&V
\end{tikzcd}\]
commutes up to homotopy, using the stacked combination of the homotopies
\begin{align*}
&W\xr{\Phi_t^{UW}|}U\but X\xr{G|\ }V\but Y\subset V,\\
&W\xr{\,\phi_1^W|\,}W\but X\xr{H_t|}V\but Y\subset V,\\
&W\xr{\Phi_t^{W\tilde U}|}\tilde U\but X\xr{\tilde G|\ }V\but Y\subset V,
\end{align*}
where $H_t(x)=H(x,t)$.
Thus $\G$ and $\tilde\G$ represent the same $\pro\Ho$-morphism $\cev\U\to\cev\V$.%
\footnote{In fact, the $\pro\Ho$-morphism determined by $\G$ does not depend 
on the choice of the homotopies $\phi_t^U$ and so depends only on the approaching
homotopy class of $G$ (this is not used in the proof of the theorem).
Indeed, suppose that $\psi_t^U$ are different homotopies of the same kind.
Let us define a homotopy $\chi_t^U\:M\to M\but X$ between $\phi_1^U$ and
$\psi_1^U$ by $\chi_t^U(x)=\psi_t^U\big(\phi_{1-t}^U(x)\big)$.
Then $U\xr{\chi_t^U|}U\but X\xr{G|}V\but X\subset V$ is a homotopy
between $g_V$ and the similar map defined using $\psi_t$ in place of $\phi_t$.}

Now let $f\:X\to Y$ be a continuous map. 
By Lemma \ref{Milnor}(a) it extends to a continuous map $F\:M\to N$ such that $F^{-1}(Y)=X$.
Let $\F=(f_V)_{V\in\V}$ be given by the preceding construction applied to
the restriction $M\but X\to N\but Y$ of $F$.
Given a $V\in\V$, let $U=F^{-1}(V)$.
Then the diagram
\[\begin{tikzcd}
X\ar[r,"f"]\ar[d,hook]&Y\ar[d,hook]\\
U\ar[r,"f_V"]&V
\end{tikzcd}\]
commutes up to homotopy, using the homotopy $X\xr{\phi_t^U|}U\xr{F|}V$.
Hence $\F$ is an expansion of $f$.

In particular, the preceding discussion applies to the case where $f=\id_X$ 
(but $M\ne N$).
From this it follows that if given approaching homotopy classes $[G]\:M\but X\to N\but X$ 
and $[G']\:M'\but X\to N'\but Y$ represent the same fine shape morphism, then
the corresponding $\pro\Ho$-morphisms $[\G]\:\cev\U\to\cev\V$ and 
$[\G']\:\cev{\U'}\to\cev{\V'}$ represent the same shape morphism.
Thus we obtain a map $\F_{XY}\:[X,Y]_\Fish\to[X,Y]_\Sh$.
By the general case of the preceding discussion it sends the fine shape morphism 
represented by a map $f\:X\to Y$ to the shape morphism represented by $f$.
Thus the forgetful map $[X,Y]\to [X,Y]_\Sh$ factors through $\F_{XY}$.
Also we get that $\F_{XX}([\id_X])=[\id_X]$.
It is straightforward to verify that for fine shape morphisms $X\xr{f}Y\xr{g}Z$ we have
$\F_{YZ}(g)\F_{XY}(f)=\F_{XZ}(gf)$.
\end{proof}

Theorem \ref{shape}(a) immediately yields

\begin{corollary} Fine shape equivalence implies shape equivalence.
\end{corollary}

\subsection{Strong shape}

We refer to \cite{Mard} for a description of the strong shape category for general metrizable spaces.
This description is the one that we will follow in the proof of the next theorem.
The set of strong shape morphisms between metrizable spaces $X$ and $Y$ is denoted by $[X,Y]_\Ssh$.

\begin{theorem}\label{strongshape}
(a) There is a functor $\F^\Fish_\Ssh$ from the fine shape category to the strong shape category 
which sends every object to itself.

(b) The forgetful functor from the homotopy category of metrizable spaces to the strong shape category 
is the composition of $\F^\Ho_\Fish$ and $\F^\Fish_\Ssh$.
\end{theorem}

Let us note that Theorem \ref{shape} is a consequence of Theorem \ref{strongshape}.

\begin{proof}
We will use the notation introduced in the proof of Theorem \ref{shape}.
According to \cite{Mard}, first of all we need to come up with ANR resolutions of
$X$ and $Y$ indexed by cofinite directed posets (a poset $P$ is called {\it cofinite} 
if for each $\lambda\in P$ there exist only finitely many $q\in P$ such that $q\le p$).%
\footnote{The logic behind this requirement, which might be not obvious from \cite{Mard},
is that in order to vastly simplify the higher coherence conditions (which distinguish 
strong shape from shape), one wants to have the maps between the indexing sets to be 
monotone.
In our case they are already monotone; so it might be possible to avoid this step by
using a different version of the strong shape category (but the author is not aware 
of any appropriate candidates in the literature).}
To this end let $\V^+$ consist of all finite subsets $S$ of $\V$ that have a minimal 
element (with respect to the order by inclusion), denoted $\min S$.
$\V^+$ is ordered by inclusion and is easily seen to be cofinite and directed.
Let $\cev\V^+$ be the inverse system indexed by $\V^+$ and consisting of the sets 
$V^+_S\bydef V_{\min S}$, with bonding maps $V^+_T\to V^+_S$, $S\subset T$, given by 
the inclusions $V_{\min T}\subset V_{\min S}$.
Then $\cev\V^+$ is an ANR resolution of $Y$ (see \cite{Mard}*{Lemma 6.31}).
Let $\cev\U^+=(U^+_S,\subset;\U^+)$ be the similarly defined ANR resolution of $X$.

Let $G\:M\but X\to N\but Y$ be an approaching map.
Given a finite subset $S\subset\V$, let $U^\star_S=X\cup\bigcap_{V\in S} G^{-1}(V\but Y)$.
Let $G^+\:\V^+\to\U^+$ be defined by $G^+(S)=\{U^\star_T\mid T\subset S,\, T\ne\emptyset\}$; 
note that $G^+(S)$ does have a minimal element, namely, $U^\star_S$.
Clearly $G^+$ is monotone.
Given an $S\in\V^+$, let $V\bydef V^+_S=V_{\min S}$ and $U\bydef U^+_{G^+(S)}=U^\star_S$.
Then we have the composition $g^+_S\:U\xr{\phi_1^U}U\but X\xr{G|}V\but Y\subset V$.
Similarly to the proof of Theorem \ref{shape}, $\G^0\bydef \big(g^+_S)_{S\in\V^+}$ is 
an $\inv\Ho$-morphism between $\cev\U^+$ and $\cev\V^+$.
(Using this $\G^0$ in place of $\G$ we can now reprove Theorem \ref{shape} by repeating 
the previous constructions.)

Let $\Delta(\V^+)$ be the order complex of the poset $\V^+$; thus an element 
$\sigma\in\Delta(\V^+)$ is a chain of subsets $S_0\subset\dots\subset S_n$ of $\V$,
where each $S_i$ belongs to $\V^+$, i.e.\ has a minimal element.
Let $V_i\bydef V^+_{S_i}=V_{\min S_i}$ and $U_i\bydef U^+_{G^+(S_i)}=U^\star_{S_i}$.
Thus $V_0\subset\dots\subset V_n$ and $U_0\subset\dots\subset U_n$.
As before we have the composition 
$g^+_{S_i}\:U_i\xr{\phi_1^{U_i}|}U_i\but X\xr{G|}V_i\but Y\subset V_i$.
We want to define a simplex of homotopies between appropriate restrictions of these maps.
To this end we first define a simplex of homotopies $\Phi_t^{U_0\dots U_n}\:M\to M\but X$
between the maps $\phi_1^{U_i}\:U_i\to U_i\but X$.
Namely, given a point $t=(t_0,\dots,t_n)\in\Delta^n\subset\R^{n+1}$ (in other words,
$t_0+\dots+t_n=1$ and each $t_i\ge 0$), we set 
$\Phi_t^{U_0\dots U_n}(x)=\phi_{t_n}^{U_n}\circ\dots\circ\phi_{t_0}^{U_0}(x)$.
Let us note that $\Phi_t^{U_0\dots U_n}|_{M\but U_n}=\id_{M\but U_n}$ for all $t\in I$.
Now let us consider the $\Delta^n$-homotopy 
$g^\sigma_t\:U_0\xr{\Phi_t^{U_0\dots U_n}|}U_n\but X\xr{G|}V_n\but Y\subset V_n$
between the maps $g^+_{S_i}$.
Clearly it is natural with respect to inclusions of faces of the form $\tau\subset\sigma$.
Hence $\G^+\bydef (g^\sigma_t)_{\sigma\in\Delta(\V^+)}$ is a coherent mapping $\cev\U^+\to\cev\V^+$
in the sense of \cite{Mard}.

Now let $H\:(M\but X)\x I\to N\but Y$ be an approaching homotopy between
approaching maps $G,\tilde G\:M\but X\to N\but Y$.
A coherent mapping $\tilde\G^+\:\cev\U^+\to\cev\V^+$ is defined similarly to $\G^+$.
Given a $\sigma=(S_0\subset\dots\subset S_n)\in\Delta(\V^+)$, the corresponding components 
of $\G^+$ and $\tilde\G^+$ are $g^\sigma_t\:U_0\to V_n$ (defined above) and 
$\tilde g^\sigma_t\:\tilde U_0\to V_n$ (defined similarly).
By Lemma \ref{approaching homotopy} each $H^{-1}(V\but Y)\cup X$ contains $W\x I$ 
for some $W\in\U$.
Hence each $X\x I\cup\bigcap_{V\in S} H^{-1}(V\but Y)$ contains $W^\star_S\x I$ for some 
$W^\star_S\in\U$.
Let $W_i=W^\star_{S_i}$.
Clearly, $W_0\subset\dots\subset W_n$ and each $W_i\subset U_i\cap\tilde U_i$.
Then the diagram of $\Delta^n$-homotopies
\[\begin{tikzcd}
W_0\ar[r,hook]\ar[d,hook]&U_0\ar[d,"g^\sigma_t"]\\
\tilde U_0\ar[r,"\tilde g^\sigma_t"]&V_n
\end{tikzcd}\]
commutes up to $\Delta^n\x I$-homotopy, using the stacked combination $h^\sigma_{(t,s)}$ of 
the $\Delta^n\x I$-homotopies
\begin{align*}
&W_0\xr{\Phi_{st_1+(1-s)t_2}^{U_0\dots U_nW_0\dots W_n}|}U_n\but X\xr{G|\ }V_n\but Y\subset V_n,\\
&W_0\xr{\ \ \ \Phi_t^{W_0\dots W_n}|\ \ \ }W_n\but X\xr{H_s|}V_n\but Y\subset V_n,\\
&W_0\xr{\Phi_{st_1+(1-s)t_2}^{W_0\dots W_n\tilde U_0\dots\tilde U_n}|}\tilde U_n\but X\xr{\tilde G|\ }V_n\but Y\subset V_n,
\end{align*}
where $H_s(x)=H(x,s)$ and $t_1$ and $t_2$ are the images of $t\in\Delta^n$ in the two copies
of $\Delta^n$ lying in the join $\Delta^n*\Delta^n=\Delta^{2n+1}$.
Clearly this $h^\sigma_{(t,s)}$ is natural with respect to 
inclusions of faces of the form $\tau\subset\sigma$.
Thus $\G^+$ and $\tilde\G^+$ are coherently homotopic.

Now let $f\:X\to Y$ be a continuous map. 
It extends to an $F\:M\to N$ such that $F^{-1}(Y)=X$.
Let $\F=(f^\sigma_t)_{\sigma\in\Delta(\V^+)}$ be the coherent mapping
$\cev\U^+\to\cev\V^+$ given by the preceding construction applied to
the restriction $M\but X\to N\but Y$ of $F$.
Given a subset $S\subset\V$, let $U^\star_S=\bigcap_{V\in S} F^{-1}(V)$.
Given a $\sigma=(S_0\subset\dots\subset S_n)\in\Delta(\V^+)$, let 
$V_i\bydef V^+_{S_i}=V_{\min S_i}$ and $U_i\bydef U^\star_{S_i}$.
Then the diagram
\[\begin{tikzcd}
X\ar[r,"f"]\ar[d,hook]&Y\ar[d,hook]\\
U_0\ar[r,"f^\sigma_t"]&V_n
\end{tikzcd}\]
commutes up to $\Delta^n\x I$-homotopy, using the $\Delta_n\x I$-homotopy 
\[\xi^\sigma_{(t_1,\dots,t_n,s)}\:U_0\xr{\phi_{st_n}^{U_n}\circ\dots\circ\phi_{st_0}^{U_0}|}
U_n\xr{F|}V_n.\]
Clearly this homotopy is natural with respect to inclusions of faces of the form $\tau\subset\sigma$.
Hence $\F$ is a {\it coherent expansion} of $f$, in the sense that the following diagram of 
coherent maps (where $X$ and $Y$ are understood as inverse systems indexed by the singleton, 
and the vertical arrows are the inverse cones of the inclusion maps)
\[\begin{CD}
X@>f>>Y\\
@VVV@VVV\\
\cev\U^+@>\F>>\cev\V^+,
\end{CD}\]
commutes up to coherent homotopy.

The remainder of the proof is similar to that of Theorem \ref{shape}.
\end{proof}

Theorem \ref{strongshape}(a) immediately yields

\begin{corollary} Fine shape equivalence implies strong shape equivalence.
\end{corollary}

\subsection{Sequential strong shape}

In view of Theorem \ref{fish-sash} one could expect that fine shape
coincides with strong shape for coronated ANRs (this would indeed be a dual theorem).
But this turns out to be false.

\begin{example} \label{comb-example}
Let $X\subset\R^2$ be the comb-and-flea set
$\{\frac1n\mid n\in\N\}\x[0,1]\cup (0,1]\x\{1\}\cup\{(0,0)\}$.
It is not hard to see that $X$ is a coronated ANR \cite{M-V}.
However, the forgetful map $[pt,X]_\Fish\to[pt,X]_\Ssh$ is not surjective.

Indeed, there exist strong shape morphisms $pt\to X$ (in fact, uncountably 
many of them) which do not factor through strong shape morphisms of the form $pt\to K$, where 
$K\subset X$ is compact \cite{M00}*{Example \ref{book:ssh-ash}}.
On the other hand, every fine shape morphism $pt\to X$ obviously factors through a 
fine shape morphism of the form $pt\to K$, where $K\subset X$ is compact.
\end{example}

\begin{remark} \label{non-example}
By the short exact sequence of \cite{M-V} it is not possible to construct such 
an example using Steenrod--Sitnikov homology or more generally any additive homology theory 
satisfying the map excision axiom.
(By Theorem \ref{fs-me} every fine shape invariant homology theory satisfies the map excision axiom.)
So the difference between fine shape and strong shape for coronated ANRs appears to be
an unstable phenomenon.
It is natural to conjecture that stable fine shape (which has yet to be defined) coincides
with stable strong shape (see \cite{Mard} and references there) for coronated ANRs.
\end{remark}

\subsection*{Acknowledgements}

I'm grateful to F. Ancel and V. Zemlyanoy for stimulating discussions and useful remarks.

\subsection*{Disclaimer}

I oppose all wars, including those wars that are initiated by governments at the time when 
they directly or indirectly support my research. The latter type of wars include all wars 
waged by the Russian state in the last 25 years (in Chechnya, Georgia, Syria and Ukraine) 
as well as the USA-led invasions of Afghanistan and Iraq.

\end{document}